\documentclass[a4paper, 11pt, english, fleqn, leqno, final]{siamltex704}
\usepackage[english]{babel}
\usepackage[latin1]{inputenc}
\usepackage[T1]{fontenc}
\usepackage{amsmath}
\usepackage[ruled]{algorithm2e} % ubuntu package: texlive-science
\usepackage{subfigure}
\usepackage{varwidth,xcolor}
\usepackage{graphicx}
\usepackage{amssymb}
\def\^{\widehat}

\newcommand{\norm}[1]{\Vert #1 \Vert}

\def\phi{\varphi}

\numberwithin{equation}{section}

\renewcommand{\phi}{\varphi}

\def\~{\widetilde}
\def\^{\widehat}
\newcommand{\ee}{{\rm e}\hspace{1pt}}

\newcommand{\dd}{\hspace{1pt}{\rm d}\hspace{0.5pt}}
\newcommand{\abs}[1]{\left| #1 \right|}

\newcommand{\RR}{\mathbb{R}}

\newcommand{\CC}{\mathbb{C}}

\newcommand{\NN}{\mathbb{N}}

\renewcommand{\vec}{\operatorname{vec}}
\newcommand{\err}{\operatorname{err}}
\newcommand{\errT}[1]{\operatorname{err}_{{\rm T},#1}}

\newcommand{\errK}[1]{\operatorname{err}_{{\rm K},#1}}
\newcommand{\errKt}[1]{\widetilde{\operatorname{err}}_{{\rm K},#1}}
\newcommand{\eps}{\varepsilon}
\newcommand{\veps}{\varepsilon}

\makeatletter
\def\Ddots{\mathinner{\mkern1mu\raise\p@
\vbox{\kern7\p@\hbox{.}}\mkern2mu
\raise4\p@\hbox{.}\mkern2mu\raise7\p@\hbox{.}\mkern1mu}}
\makeatother

\newtheorem{thm}{Theorem}

\newtheorem{lem}[thm]{Lemma}
\newtheorem{cor}[thm]{Corollary}

\newcommand*\samethanks[1][\value{footnote}]{\footnotemark[#1]}

\title{%An exponential integrator for
Krylov approximation of
ODEs with \\ polynomial parameterization
}

\author{Antti Koskela\thanks{Department of Mathematics, Royal Institute of Technology (KTH),
Stockholm, SeRC Swedish e-Science Research Center, \texttt{akoskela@kth.se}, \texttt{eliasj@kth.se}.
}
\and Elias Jarlebring\samethanks
\and Michiel E. Hochstenbach\thanks{Department of Mathematics and Computer Science,
TU Eindhoven,
PO Box 513,
5600 MB Eindhoven,
The Netherlands
\texttt{www.win.tue.nl/$\sim$hochsten/}.
This author was supported by an NWO Vidi research grant.}}

    \date{}

\begin{document}
   \maketitle
\begin{abstract}
%Suppose $A_0,\ldots,A_N\in\CC^{n\times n}$ are given matrices.
We propose a new numerical method  to solve linear
ordinary differential equations of the type
$\frac{\partial u}{\partial t}(t,\veps) =
A(\veps) \, u(t,\veps)$,
where $A:\CC\rightarrow\CC^{n\times n}$ is a
matrix polynomial with large and sparse matrix coefficients.
The algorithm computes an explicit parameterization
of approximations of $u(t,\veps)$ such that approximations
for many different values of $\veps$ and $t$ can be obtained
with a very small additional computational effort.
The derivation of the algorithm
is based on a reformulation of
the parameterization as a linear parameter-free
ordinary differential equation and on approximating the product
of the matrix exponential and a vector with a Krylov method.
The Krylov approximation is generated with Arnoldi's method
and the structure of the coefficient matrix turns out to have an
independence on the truncation parameter so that
it can also be interpreted as Arnoldi's method applied to an infinite dimensional matrix.
We prove the superlinear convergence of the algorithm and provide a posteriori error
estimates to be used as termination criteria. The behavior of the algorithm
is illustrated with examples stemming from spatial discretizations
of partial differential equations.
\end{abstract}
\begin{keywords}
Krylov methods, Arnoldi's method, matrix functions,
matrix exponential,
exponential integrators,
parameterized ordinary differential equations,
Fréchet derivatives, %moment matching
%parameterized partial-differential equations,
%parameterized uncertainty,
model order reduction.
\end{keywords}

\begin{AMS}
65F10, 65F60, 65L20, 65M22
\end{AMS}

\begin{DOI}

\end{DOI}

\section{Introduction}

Let $A_0$, $A_1$, \dots, $A_N\in \mathbb{C}^{n \times n}$ be given matrices and consider the %following
parameterized linear time-independent ordinary differential equation
\begin{equation} \label{eq:IVP}
\frac{\partial u}{\partial t}(t,\veps) =
A(\veps) \, u(t,\veps), \quad
%\Big(\sum\limits_{\ell = 0}^N \eps^\ell A_\ell \Big) \, u(t,\veps), \quad
u(0,\veps) = u_0,
\end{equation}
where $A$ is the matrix polynomial $A(\veps):=A_0+\veps A_1+\cdots+\veps^NA_N$.
Although most  of our results are general,
the usefulness of the approach is more explicit in
  a setting where $N$ is not very large and
the matrices $A_0$, \dots, $A_N$ are large and sparse, e.g.,
stemming from a spatial finite-element semi-discretization of a
parameterized partial-differential
equation of evolutionary type.
%$u_0 \in \mathbb{C}^n$ and $\eps > 0$.\\
%Our goal is an efficient numerical method to compute the solution of the initial value problem

We present a new iterative algorithm for the parameterized ODE \eqref{eq:IVP},
which gives an explicit parameterization of the solution. This parameterization
is explicit in the sense that after executing the algorithm
we can find a solution to the ODE \eqref{eq:IVP}
for many different values of $\varepsilon$ and $t>0$
without essential additional computational effort.
Such explicit parameterizations of solutions
are useful in various settings, e.g., in parametric model order reduction and
in the field of uncertainty quantification (with a single model parameter);
see the discussion of model reduction below and the references in  \cite{Benner:2013:MOR}.
%See  references in
%\cite{Bui-Thanh:2008:PARAMETRIC,Benner:2014:MOR}.
% Such an explicit parameterization is common in many applications, e.g., when the $\varepsilon$ is an optimization parameter, e.g.,
%in the setting of identification . See XXX for
%recent results in identification (Ljung or people in Linköping?).
%and the objective function can only be expressed
%as the solution to an ODE.
%A similar situation occurs when
%the parameter $\veps$ is a model uncertainty and
%predictions must be done with a large number of values of $\veps$,
%e.g., with a statistical approach.
%See XX for literature on uncertainty quantifications.

%UQ: http://web.stanford.edu/group/cits/pdf/lectures/barth.pdf
%Parameter selection: http://www.diva-portal.se/smash/get/diva2:754702/FULLTEXT01.pdf

%for several values of $\varepsilon$ and $t$. \\

%Note that the
%can be explicitly expressed with a matrix exponential
%of a matrix containing $\varepsilon$.
The parameterization of the solution is represented as follows.
Let the coefficients of the Taylor expansion of the solution with
respect to the parameter $\varepsilon$ be denoted
by $c_0(t)$, $c_1(t)$, \dots, i.e.,
\begin{equation} \label{eq:taylor_series}
u(t,\veps) =
%\exp\left( t \, \Big(\sum\limits_{\ell = 0}^N \eps^\ell A_\ell \Big) \right) \, u_0
\exp(t A(\veps)) \, u_0
= \sum\limits_{\ell=0}^\infty \eps^\ell c_\ell(t).
\end{equation}
As $\, \exp\left( t A(\veps) \right) \,$ is an entire function of a matrix polynomial,
the expansion \eqref{eq:taylor_series} exists for all $\eps \in \mathbb{C}$.

% Throughout this paper we assume that the expansion exists for those
% choices of $\varepsilon$ which are under consideration.
% For instance, the expansion exists
% in a neighborhood around $\veps=0$ if all eigenvalues of
% $A(0)$ are distinct,
% or if $A(\varepsilon)$ is symmetric for all $\varepsilon\in\RR$.
% $A(0)$ is diagonalizable and $\varepsilon$
%is sufficiently small such that the expansion \eqref{eq:taylor_series} exists.

Consider the approximation stemming from the
truncation of the Taylor series \eqref{eq:taylor_series} and
a corresponding approximation of the Taylor coefficients
\begin{align}
u_k(t,\veps) &:= \sum\limits_{\ell=0}^{k-1} \eps^\ell c_\ell(t)\notag \\
&\approx \sum\limits_{\ell=0}^{k-1} \eps^\ell \widetilde{c}_\ell(t) =: \widetilde{u}_k(t,\veps).
\label{eq:approximative_series}
\end{align}
Our approach gives an explicit parameterization with respect $t$
of the approximate coefficients $\widetilde{c}_0(t)$,\dots,$\widetilde{c}_{k-1}(t)$
which, via \eqref{eq:approximative_series},
gives an approximate solution with an explicit parameterization
with respect to $\varepsilon$ and $t$.

The derivation of our approach is based on an explicit characterization of the time-dependent
coefficients $c_0(t)$,\dots,$c_{m-1}(t)$. We
prove in Section~\ref{sect:algder} that they are solutions
to the linear ordinary differential equation of size $nm$,
\begin{equation}\label{eq:larger_ODE}
\frac{\dd}{\dd t}\begin{bmatrix}
c_0(t)\\
\vdots\\
c_{m-1}(t)
\end{bmatrix}=L_m
\begin{bmatrix}
c_0(t)\\
\vdots\\
c_{m-1}(t)
\end{bmatrix},\qquad
\begin{bmatrix}
c_0(0)\\
\vdots\\
c_{m-1}(0)
\end{bmatrix}=
\begin{bmatrix}
u_0\\
0\\
\vdots\\
0
\end{bmatrix}.
\end{equation}
The matrix $L_m$ in \eqref{eq:larger_ODE}
is a finite-band block Toeplitz matrix and also
a lower block triangular matrix.

Since \eqref{eq:larger_ODE} is a standard linear ODE,
we can in principle
apply any numerical method to compute the solution which results
in approximate coefficients $\widetilde{c}_0(t)$,\dots,$\widetilde{c}_{m-1}(t)$.
%Exponential integrators is a recent class
%of methods for differential equations, which combine
%techniques for Krylov-methods with approximations of matrix functions.
Exponential integrators combined with Krylov approximation of matrix functions
have recently turned out to be an efficient class of methods for large-scale
(semi)linear ODEs arising from PDEs \cite{Hochbruck:1997:KRYLOV,Hochbruck:1998:EXPINT}.
See also \cite{Hochbruck:2010:EXPINT}
for a recent summary of exponential integrators.
Krylov approximations of matrix functions have a feature which is
suitable in our setting: after one run they give
parameterized approximations with respect to the time-parameter.

Our derivation is based on approximating the solution of \eqref{eq:larger_ODE},
i.e., a product of the matrix exponential and a vector, using a Krylov method.
This is done by exploiting the structure of the coefficient matrix $L_m$.
We show that when we apply Arnoldi's method to construct a Krylov
subspace corresponding to \eqref{eq:larger_ODE}, the
block Toeplitz and lower block triangular property of $L_m$
result in a particular structure in the basis matrix given by
Arnoldi's method.

The structure of $L_m$ is such that, in a certain sense, the algorithm
can be equivalently extended to infinity. For example when $N=1$,
the basis matrix is extended with one block row
as well as a a block column in every iteration.
This is analogous to the infinite Arnoldi method which has been
developed for nonlinear eigenvalue problems \cite{Jarlebring:2012:INFARNOLDI}
and linear inhomogeneous ODEs \cite{Koskela:2014:INFARNSEMILIN}.
This feature implies that the algorithm
does not require an a priori choice of the truncation parameter $m$.

%We apply an exponential integrator to the ODE

%* Via a theoretical reformulation we can compute coeffs from larger matrix *

%* Exponential integrator on larger matrix *

%* Expansion is dynamic in the sense of infinite Arnoldi method *

We prove convergence of the algorithm (in Section~\ref{sect:apriori})
and also provide % further insight and
a termination criteria by giving a posteriori error estimates in Section~\ref{sect:aposteori}.

%The coefficients $c_0,\ldots,c_{m-1}$ are not
%directly available, but we show how to form
%an approximation of the coefficients, such that
%our total approximation is
%\begin{equation}\label{eq:approximative_series}
%u(t,\veps)\approx \widetilde{u}_m(t,\veps)=\sum\limits_{\ell=0}^{m-1} \eps^\ell \widetilde{c}_\ell(t),
%\end{equation}
%
%
%* further description of the method *

The results can be interpreted and related
to other approaches from a number of different
perspectives. From one viewpoint, our result is related to
recent work on computations and theory for Fréchet derivatives
of matrix functions, e.g.,
\cite{Higham:2014:HIGHFRECHET, Najfeld:1995:DERIVATIVES,Mathias:1996:CHAIN}.
%, \cite{Saad:1992:MATEXP}
%which is used several
%times in the analysis of our approach.
As an illustration of a relation, consider the
special case $N=1$.
The first-order expansion of the matrix
exponential in  \eqref{eq:taylor_series}
and \cite[Chapter 3.1]{Higham:2008:MATFUN} gives
\begin{align*}
u(t,\varepsilon) & =\exp(t(A_0+\veps A_1)) \, u_0 \\[1mm]
& =\exp(tA_0) \, u_0+ L_{\exp}(tA_0,\veps tA_1) \, u_0+o(|\veps| \, |t| \, \|A_1\|),
\end{align*}
where $L_{\exp}$ is the Fréchet derivative of the matrix exponential.
Since the Fréchet derivative is linear in the second parameter,
the first coefficient is explicitly given by $c_1(t)=L_{\exp}(tA_0,tA_1) \, u_0$.
The higher order terms $c_2,c_3,\ldots$ have corresponding
relationships with the higher order Fréchet derivatives.
An analysis of higher order Fréchet derivatives is given in
\cite{Relton:2014:PHD}.
In contrast to the current Fréchet derivative approaches,
our approach is an iterative Krylov method with a
focus on large and sparse matrices and a specific starting vector,
which unfortunately does not appear to be easily constructed within the Fréchet derivative
framework.
%, such that the parameter
%$m$ does not have to be chosen a priori.
%Moreover,
%our approach is specialized to the computation
%of the expansion for many $\varepsilon$
%and does not involve an a priori decision of
%the number of terms that will be needed.
%

%Exponential integrators  Krylov approximation
%of the matrix exponential function  * short description of
%exponential integrators *
%

The general approach to compute parameterized solutions
to parameterized problems is very common in the field
of model order reduction (MOR).
See the recent survey papers \cite{Benner:2013:MOR,Antoulas:2006:MOR}.
In the terminology of MOR, our approach can be
 interpreted as a time-domain model order reduction technique
for parameterized linear dynamical systems,
without input or output. Parametric
MOR is summarized in \cite{Benner:2013:MOR}; see also \cite{Lietaert:2015:INTERPOLATORY,Yue:2012:KRYLOVPADE}.
Our approach is a Krylov method to compute a moment matching
approximation in the model parameter $\veps$.
%Moment matching model order reduction techniques have
%received considerable attention.
%However, not many of the methods for parametric MOR
%can be interpreted as moment matching in the time-domain.
%the series \eqref{eq:taylor_series}
%can be referred to as a moment matching technique
%in the model parameter.
There are time-domain Krylov methods, e.g.,
those described in PhD thesis~\cite{Eid:2008:TIMEDOMAIN}.
To our knowledge, none of these methods  can be interpreted
as exponential integrators.
%These approaches are conceptually different from ours,
%since they take the input-output behavior into account, and do
%not consider the initial condition. On the other hand,
%our approach does not take an input and is specialized
%for a given initial condition.
%Moreover, our approach does not require an a priori choice
%of the number of terms required in the expansion.
%
%In contrast
%to our approach, there are
%many model order techniques that
%take the input-output relation into account,
%typically frequency domain or by using Graminas, e.g. Bauer, Gugercin, Benner, Stykel, Meerbergen, Sandberg, .

We use the following notation in this paper.
We let $\#S$ denote the number of elements
in the set $S$, and $\vec(B)$ denote vectorization, i.e.,
$\vec(B)=[b_1^T,\ldots,b_k^T]^T\in\CC^{nk}$, where $B=[b_1,\ldots,b_k]\in\CC^{n\times k}$.
By $I_n$ we indicate the identity matrix of dimension $n$.
The set of eigenvalues of a matrix $A$ is denoted by $\Lambda(A)$
and the positive integers by $\NN_+$.
The logarithmic norm (or numerical abscissa) $\mu:\CC^{n\times n}\rightarrow \RR$
is defined by
\begin{equation}  \label{eq:lognorm}
\mu(A) := \max \left\{ \lambda\in\RR \, : \, \lambda \in \Lambda \left(\frac{A+A^*}{2}\right) \right\}.
\end{equation}

\section{Derivation of the algorithm}\label{sect:algder}

\subsection{Representation of the coefficients using the matrix exponential}

To derive the algorithm, we first show
that the time-dependent coefficients $c_0(t)$, \dots, $c_{m-1}(t)$
are solutions to a linear time-independent ODE of the form \eqref{eq:larger_ODE}, i.e.,
they are explicitly given by the matrix exponential.

%We note that similar characterizations  form a basis of the condition number est
%The result can be seen as a variation of \cite[Theorem~4.1]{Mathias:1996:CHAIN}, but with a block lower-triangular matri
% (instead of a block upper-triangular matrix), which turns out to be
%an important aspect in the algorithm construction.

\begin{thm}[Explicit formula with matrix exponential] \label{thm:block_toeplitz}
The Taylor coefficients  $c_0(t), \ldots, c_{m-1}(t)$ in \eqref{eq:taylor_series}
are explicitly given by
\begin{equation}\label{eq:block_toeplitz}
\vec(c_0(t), \ldots, c_{m-1}(t))
%\begin{bmatrix} c_0(t)^\mathrm{T}  & \ldots & c_{m-1}(t)^\mathrm{T}  \end{bmatrix}^\mathrm{T}
= \exp(t L_m) \, \widetilde{u}_0,
\end{equation}
where
\begin{equation} \label{eq:L_N}
L_m  := \begin{bmatrix}
A_0    &        &        &        &        &      \\
A_1    & \ddots &        &        &        &      \\
\vdots & \ddots & \ddots &        &        &      \\
A_{\widehat{N}}    & \ddots & \ddots & \ddots &        &      \\
       & \ddots & \ddots & \ddots & \ddots &      \\
       &        &   A_{\widehat{N}}  & \hdots & A_1    &  A_0 \\
\end{bmatrix} \in \mathbb{C}^{mn \times mn}
\quad \textrm{and} \quad
\widetilde{u}_0 = \begin{bmatrix} u_0 \\ 0 \\ \vdots \\ 0 \end{bmatrix} \in \mathbb{C}^{mn},
\end{equation}
and $\widehat{N}=\min(m-1,N)$.
\end{thm}
\begin{proof}
The proof is based on explicitly forming an associated ODE.
The result can also be proven using a similar result \cite[Theorem~4.1]{Mathias:1996:CHAIN}.
We give here an alternative shorter proof for the case of the matrix
exponential, since some steps of the proof are needed in other parts of this paper.
%The differential equation \eqref{eq:c_ode} of the proof will be needed in Lemma ?
%to derive an integral formula for the vectors $c_\ell(t)$.
Differentiating \eqref{eq:taylor_series} yields that for any $j\ge 0$,
\begin{equation} \label{eq:c_coeffs}
\frac{1}{j!} \frac{\partial^j}{\partial \varepsilon^j} u(t,\eps) \bigg|_{\varepsilon=0} = c_j(t).
\end{equation}
By evaluating \eqref{eq:c_coeffs} at $t=0$,
and noting that $u(0,\varepsilon)=u_0$ is independent of $\varepsilon$,
it follows that $c_0(0) = u_0$ and $c_\ell(0) = 0$ for $\ell > 0$.
The initial value problem \eqref{eq:IVP} and the expansion of its solution \eqref{eq:taylor_series} imply that
\begin{equation} \label{eq:c_ode}
\begin{aligned}
 c_j'(t) &= \frac{1}{j!} \frac{\partial^j}{\partial \varepsilon^j} \frac{\partial}{\partial t}
u(t,\varepsilon) \bigg|_{\varepsilon=0} = \frac{1}{j!} \frac{\partial^j}{\partial \varepsilon^j}
\left( \sum\limits_{i=0}^N \varepsilon^i A_i \right) \, u(t,\varepsilon) \bigg|_{\varepsilon = 0} \\
&=  \frac{1}{j!} \frac{\partial^j}{\partial \varepsilon^j} \sum\limits_{i=0}^N \varepsilon^i A_i
\sum\limits_{\ell=0}^\infty \varepsilon^\ell c_\ell(t) \bigg|_{\varepsilon=0} \\
&=  \sum\limits_{i=0}^N \sum\limits_{\ell=0}^\infty \left( \frac{1}{j!} \frac{\partial^j}{\partial \varepsilon^j}
\varepsilon^{i+\ell} \bigg|_{\varepsilon=0} \right) A_i c_\ell(t)\\
&= \sum\limits_{i=0}^{\min(N,j)} A_i c_{j-i}(t).
\end{aligned}
\end{equation}
From \eqref{eq:c_coeffs} and \eqref{eq:c_ode} it follows that the vector
$\vec(c_0(t), \ldots, c_{m-1}(t))$
%$\begin{bmatrix} c_0(t)^\mathrm{T}  & \ldots & c_{m-1}(t)^\mathrm{T}  \end{bmatrix}^\mathrm{T} $
satisfies the linear ODE \eqref{eq:larger_ODE}
with a solution given by \eqref{eq:block_toeplitz}.
\end{proof}
%\textbf{Remark.} As can be seen from Thm. 4.1 of \cite{Mathias}, it holds that
% $
%c_\ell(t) = \frac{1}{\ell!}\frac{\partial^\ell}{\partial \varepsilon^\ell} u(t,\varepsilon) |_{\varepsilon=0}.
% $

\begin{algorithm}  % [h]%\SetLine
\caption{Infinite Arnoldi algorithm for polynomial uncertain ODEs\label{alg:arnoldi}}
\SetKwInOut{Input}{Input}\SetKwInOut{Output}{Output}
\Input{$u_0\in\CC^n$, $A_0$,\dots,$A_N \in\CC^{n\times n}$}
\Output{Matrices $Q_p\in\CC^{n(1+N(p-1))\times p}$ and $H_p\in\CC^{p\times p}$ representing
approximations of the coefficients $c_{0}$, \dots, $c_{p-1}$ via \eqref{eq:ctildedef}}
\BlankLine
\nl Let $\beta=\|u_0\|$, $Q_1=u_0/\beta$,
  $\underline{H}_0=[\ ]$ \\
\For{$\ell=1,2,\ldots,p$}{
\nl Let $x=Q(:,\ell)\in\CC^{n+(\ell-1)nN}$\\
\nl Compute $y :=\vec(y_1, \dots, y_{1+(\ell-1) N})\in\CC^{n+(\ell-1)nN}$
with \eqref{eq:matvec}\\
\nl Let $\underline Q_\ell:=\begin{bmatrix}Q_\ell\\0\end{bmatrix}\in\CC^{\ell\times (n+\ell nN)}$\\
\nl Compute $h=\underline Q_\ell^*w$\label{step:orth:h}\\
\nl Compute $y_\perp:=y- \underline Q_\ell h$\label{step:orth:wperp}\\
\nl Repeat Steps~\ref{step:orth:h}--\ref{step:orth:wperp}
if necessary\label{step:reorth}\\
\nl Compute $\alpha=\|y_\perp\|$ \\
\nl Let $\underline{H}_\ell=\begin{bmatrix}
\underline{H}_{\ell-1} &h   \\
 0 &\alpha \\
\end{bmatrix}$\\
\nl Let $Q_{\ell+1}:=[\underline Q_\ell,y_\perp/\alpha]\in\CC^{(n+(\ell-1)nN)\times (\ell+1)}$
}
\nl Let $H_p\in\RR^{p\times p}$ be the leading submatrix
of $\underline{H}_p\in\RR^{(p+1)\times p}$\\
%\nl This line is incorrect: $u(t)\approx \widetilde{u}_m=[I,0]Q_m\exp(tH_m)e_1\|b\|$
\end{algorithm}
\subsection{Algorithm}

%The characterization of the coefficients in
Theorem~\ref{thm:block_toeplitz} can be used
to compute the coefficients $c_\ell(t)$ if we can compute the matrix
exponential of $L_m$ times the vector $\widetilde{u}_0$.
We use a Krylov approximation which exploits the structure
of the problem.
See, e.g., \cite{Hochbruck:1997:KRYLOV,Hochbruck:1998:EXPINT} for
literature on Krylov approximations of matrix functions.
%of the product
%of the matrix exponential and a vector can be summarized as follows.

The Krylov approximation of $v(t)=\exp(tB)v_0$, consists of
$p$ steps of the Arnoldi iteration for the matrix $B$ initiated
with the vector $v_0$. This results in the Arnoldi relation %factorization
\[
  BQ_p=Q_{p+1}\underline{H}_p,
\]
where $\underline{H}_p\in\CC^{(p+1)\times p}$ is a Hessenberg matrix and
$Q_p$ is an orthogonal matrix spanning the Krylov subspace
$\mathcal{K}_p(B,v_0)=\operatorname{span}(v_0,Bv_0,\ldots,B^{p-1}v_0)$.
The Krylov approximation of $\exp(tB)v_0$ is given by
\[
  v(t)=\exp(tB)v_0\approx Q_p\exp(tH_p)e_1 \, \|v_0\|,
\]
where $H_p\in\CC^{p\times p}$ is the leading submatrix of $\underline{H}_p$,
and $e_1$ is the first unit basis vector.

The only way $B$ appears in the Arnoldi algorithm is in the
form of matrix vector products. Moreover, the Arnoldi
algorithm is initiated with the vector $v_0$.
Suppose we apply this Arnoldi approximation to \eqref{eq:block_toeplitz}.
In the first step we need to compute the matrix vector
product
\begin{equation}\label{eq:Lm_first_step}
  L_m\vec(u_0, 0, \dots, 0)=\vec(A_0 \, u_0, \dots, A_N \, u_0, 0, \dots, 0),
\end{equation}
which is more generally given as follows.
%The Arnoldi approximation of the product
%of the matrix exponential and a vector is given by
%\[
%  \exp(tB)b\approx Q_m\exp(tH_p)e_1\|b\|,
%\]
%where $Q$ and $\underline{H}_p$ satisfy an Arnoldi relation.
%We have the following explicit form
%for the action of $L_m$.
\begin{lem}[Matrix vector product]\label{lem:matvec}
Suppose
$x=\vec(x_1, \dots, x_j, 0, \dots, 0)=\vec(X)\in\CC^{nm}$,
where $x_1,\ldots,x_j\in\CC^n$ and $m>j+N$.
Then,
\[
L_mx=\vec(y_1,\ldots,y_{j+N},0,\ldots,0),
\]
where
\begin{equation}\label{eq:matvec}
  y_\ell=\sum_{i=\max(0,\ell-k)}^{\min(N,\ell-1)}A_ix_{\ell-i},
\;\;\ell=1,\ldots,j+N.
\end{equation}
\end{lem}
\begin{proof}
Suppose $S\in\RR^{m\times m}$ is the shift matrix
$S:=\sum_{\ell=1}^{m-1}e_{\ell+1}e_{\ell}^T$ which
 satisfies $(S^i)^T=\sum_{\ell=1}^{m-i}e_{\ell}e_{\ell+i}^T$.
We have
\begin{align*}
L_mx & =\sum_{i=0}^N (S^i\otimes A_i)\vec(X)=\sum_{i=0}^N\vec(A_iX(S^i)^T) \\
& = \sum_{i=0}^N\sum_{\ell=1}^{m-i}\vec(A_iXe_\ell e_{\ell+i}^T).
\end{align*}
Note that $Xe_\ell=x_\ell$ if $\ell\le j$ and $Xe_\ell=0$ if $\ell>j$.
Hence,
by using the assumption $m>j+N$, and reordering the terms in the sum
we find the  explicit formula $L_mx=\sum_{i=0}^N\sum_{\ell=1}^{j}\vec(A_ix_\ell e_{\ell+i}^T)
%=\sum_{i=0}^N\sum_{\ell=i+1}^{k+i}\vec(A_ix_{\ell-i} e_{\ell}^T)
=\sum_{\ell=1}^{j+N}\sum_{i=\max(0,\ell-j)}^{\min(N,\ell-1)} \vec(A_ix_{\ell-i}e_\ell^T)$.
\end{proof}

\vspace{1mm}
Since the Arnoldi method consists of applying matrix vector products and
orthogonalizing the new vector against previous vectors, we
see from \eqref{eq:Lm_first_step} that the second vector in the Krylov subspace will consist of $N+1$ nonzero blocks. Repeated application
of Lemma~\ref{lem:matvec} results in a structure where the $j$th column
in the basis matrix consists of $(j-1)N+1$ nonzero blocks, under the
condition that $m$ is sufficiently large. It is natural
to store only the nonzero blocks of the basis matrix.
By only storing the nonzero blocks, the Arnoldi
method for \eqref{eq:block_toeplitz} reduces to Algorithm~\ref{alg:arnoldi}.

Note that our construction is equivalent to the Arnoldi method and the
output of the algorithm is a basis matrix and a Hessenberg matrix
which together form the approximation of the coefficients $c_0,\ldots,c_{k-1}$
\begin{equation}\label{eq:ctildedef}
 \vec(c_0(t),\ldots,c_{k-1}(t))\approx
 \vec(\widetilde{c}_0(t),\ldots,\widetilde{c}_{k-1}(t))
:=Q_p\exp(tH_p)e_1 \, \|u_0\|,
\end{equation}
where by construction $k=N(p-1)$. The approximation of the solution
 is denoted as \eqref{eq:approximative_series}, i.e.,
\[
  \widetilde{u}_{k,p}(t):=\sum\limits_{\ell=0}^{k-1} \eps^\ell \, \widetilde{c}_\ell(t),
\]
where we added an index $p$ to stress the dependence on iteration.
%Note that the algorithm implicitly gives a choice of the
%truncation parameter $k$, since the algorithm
%will result in $\widetilde{c}_{j}(t)=0$ for $j\ge k=N(p-1)$.
%
A  feature of this construction is that the algorithm does not
explicitly dependend on $m$, such that it in a sense can be
extended to infinity, i.e., it is equivalent to  Arnoldi's
method on an infinite dimensional operator.
The result can be summarized as follows.
\begin{thm}
The following procedures generate identical results.
\begin{itemize}
  \item[(i)] $p$ iterations of Algorithm~\ref{alg:arnoldi} started with $u_0$ and $A_0$, \dots, $A_N$;
  \item[(ii)] $p$ iterations of Arnoldi's method applied to $L_m$
with starting vector $e_1\otimes u_0\in\CC^{nm}$ for any $m\ge Np$;
  \item[(iii)] $p$ iterations of Arnoldi's method applied to the infinite matrix $L_\infty$
  with the infinite starting vector $e_1\otimes u_0\in\CC^\infty$.
\end{itemize}
\end{thm}
%

%* We have a growing structure. It grows by $N$ blocks per step *

%* If we carry out $p$ steps, and $m>pN+1$, we have independence of $m$.  The
%algorithm is summarized in Algorithm~\ref{alg:arnoldi}.

%In principle, the number of iterates $p$ and
%the number of coefficients computed can be different.
%* this is incorrect *
%We here propose to take all terms available after
%$p$ iterations, i.e., $m=p$. We prove in the next
%section that the method converges for this case.
%The choice to set the number of terms equal to the
%number of iterates is not an appropriate construction if the
%error in the higher coefficients is larger than
%the smaller coefficients.

%Our construction is  supported by numerical experiments in Section~\ref{sect:examples}.

\section{A priori convergence theory}\label{sect:apriori}

%Note that our method to approximate $u(t,\varepsilon)$
%can be seen as a two-stage approach:
%\begin{itemize}
%  \item[(i)] Approximate the first $m$ coefficients, $c_0,\ldots,c_{m-1}$
%  \item[(ii)] Approximate $u(t,\varepsilon)$ with the truncated Taylor series.
%\end{itemize}
To show the validity of our  approach we now
bound the total error after $p$ iterations, which is separated into two terms as
\begin{equation} \label{eq:error_splitting}
\begin{aligned}
 \err_p(t,\veps) & :=\|u(t,\veps)-\widetilde{u}_{N(p-1),p}(t,\veps)\| \\
& \,\, \le  \errK{N(p-1),p}(t,\veps)+\errT{N(p-1)}(t,\veps),
\end{aligned}
\end{equation}
where
\begin{align}
\errK{k,p}(t,\veps)&:= \|\widetilde{u}_{k,p}(t,\veps)-u_k(t,\veps)\|\label{eq:errKdef}\\[1mm]
\errT{k}(t,\veps)&:= \|u(t,\veps)-u_k(t,\veps)\|.\label{eq:errTdef}
\end{align}
A bound of $\errK{k,p}$, which corresponds to the Krylov approximation of the
expansion coefficients $c_0,\ldots,c_{k-1}$, is given in Section~\ref{sect:apriori_krylov}
and a bound on $\errT{k}$, which corresponds to the truncation of the series,
is given in Section~\ref{sect:apriori_taylor}.
After combining the main results of Section~\ref{sect:apriori_krylov} and Section~\ref{sect:apriori_taylor},
in particular formulas \eqref{eq:krylov_bound} and \eqref{eq:remainder_bound}, we
reach the conclusion that
\begin{multline}\label{eq:totalerror_bound}
%\begin{aligned}
 \err_p(t,\veps)  \le  \\C_1(t,\varepsilon) \sum\limits_{\ell=0}^{N-1}
\frac{ C_2(t,\varepsilon)^{ p +\ell-1   } \ee^{C_2(t,\varepsilon)}} {(p +\ell -2)!} \|u_0\|
 +   2  \sqrt{ \frac{1-\abs{\eps}^{2N(p-1)}}{1-\abs{\eps}^2} } \frac{ (t\alpha)^p  \ee^{t\gamma}}{p!}\|u_0\|,
%\end{aligned}
\end{multline}
where $\alpha$ and $\gamma$ are given in \eqref{eq:alphagamma}, and $C_1(t,\varepsilon)$ and
$C_2(t,\varepsilon)$ are given in \eqref{eq:C_1_and_C_2}.
Due to the factorial in the denominator
of \eqref{eq:totalerror_bound}, for fixed $\veps$ and $t>0$,
the total error approaches zero superlinearly
with respect to the iteration count $p$.% $m\rightarrow\infty$.
\subsection{A bound on the Krylov error}\label{sect:apriori_krylov}
We first study the error generated by the Arnoldi method to approximate the coefficients
$c_0(t)$,\dots,$c_{k-1}(t)$.
%These approximations $\widetilde{c}_0(t),\ldots,\widetilde{c}_m(t)$
%are now given by
%(see \eqref{eq:approximative_series})
We define
\begin{equation}\label{eq:Em}
  E_{k,p}(t)=[c_0(t),\ldots,c_{k-1}(t)]-[\widetilde{c}_0(t),\ldots,\widetilde{c}_{k-1}(t)].
\end{equation}
where $\widetilde{c}_0$, \dots, $\widetilde{c}_k$ are the
approximations given by the Arnoldi method, i.e., by the vector
\[
\widehat{c}_k(t) := \begin{bmatrix} \widetilde{c}_0(t) \\ \vdots \\ \widetilde{c}_{k-1}(t) \end{bmatrix}
= Q_p \exp(t H_p) e_1 \, \norm{u_0}.
\]
Using existing bounds for the Arnoldi approximation of the matrix exponential~\cite{Gallopoulos:1992:EFFICIENT},
we get a bound for the error of this approximation, as follows.
\begin{lem}[Krylov coefficient error bound]\label{lem:krylov_coefficient_bound}
Let $t>0$, $A_0$,\dots,$A_N \in \CC^{n\times n}$,
and $u_0\in\CC^n$. Let $\widetilde{c}_0(t)$,\dots,$\widetilde{c}_{k-1}(t)$ be the
result of Algorithm~\ref{alg:arnoldi},
and let $E_{k,p}(t)$ be defined by \eqref{eq:Em}.
Then, the total error  in the coefficients $\|\vec(E_k(t))\|$ satisfies
\begin{equation*} % \label{eq:}
 \|\vec(E_{k,p}(t))\|=\| \exp(tL_k )\widetilde{u}_0-\widehat{c}_k(t) \|
\le 2 \, \frac{ (t \alpha)^p}{p!} \, e^{t\max\{1,\beta\}} \, \|u_0\|
\end{equation*}
where
% \begin{subequations}\label{eq:alphagamma}
% \begin{eqnarray}
%   \alpha&=&\sum_{\ell=0}^N\|A_\ell\|\label{eq:normbound}\\
%   \beta&=&\mu(A_0)+\sum_{\ell=1}^N\|A_\ell\|,
% \end{eqnarray}
% \end{subequations}
\begin{equation}\label{eq:alphagamma}
\alpha=\sum_{\ell=0}^N\|A_\ell\| \quad \quad \textrm{and} \quad \quad \beta=\mu(A_0)+\sum_{\ell=1}^N\|A_\ell\|,
\end{equation}
and $\mu(B)$ denotes the logarithmic norm defined
in \eqref{eq:lognorm}
% i.e.,
%$\mu(B) = \max \{ \lambda \, : \, \lambda \in \Lambda(\frac{B+B^*}{2}) \}$.
\end{lem}
\begin{proof}
The result follows directly from~\cite[Thm.\;2.1]{Gallopoulos:1992:EFFICIENT},
and Lemma~\ref{lem:norm_bound}
%(norm bound for $ L_m $)
and Corollary~\ref{cor:log_norm_bound} in Appendix~A.
% (logarithmic norm bound for $L_m $).
\end{proof}

%* A theorem with a bound on $\errK{m}$ *
The coefficient error bound in Lemma~\ref{lem:krylov_coefficient_bound},
implies the following bound on the error $\errK{k,p}$, via the relation
\begin{equation}  \label{eq:errKmE}
  \errK{k,p}(t,\varepsilon)=\|E_{k,p}(t)[1,\varepsilon,\ldots,\varepsilon^{k-1}]^T\|.
\end{equation}
%The following theorem gives now a bound for the first term in the error splitting \eqref{eq:error_splitting}.
\begin{thm}[Krylov error bound] \label{thm:krylov_error}
Let $\errK{k,p}$ be defined in \eqref{eq:errKdef}
% $\widetilde{c}_1(t),\ldots,\widetilde{c}_m(t)$ be the
corresponding to applying $p$ steps of Algorithm~\ref{alg:arnoldi} to $A_0$,\dots,$A_N\in\CC^{n\times n}$ and $u_0\in\CC^n$. Then,
\begin{equation}  \label{eq:krylov_bound}
  \errK{k,p}(t,\veps)\le
%2 \frac{1-\eps^m}{1-\eps} \frac{ \norm{tL_m }^m \ee^{\omega(tL_m )}}{m!} \leq
 2 \, \sqrt{ \frac{1-\abs{\eps}^{2k}}{1-\abs{\eps}^2} } \ \frac{ (t\alpha)^p  \ee^{t\max\{1,\beta\}  }}{p!} \, \norm{ u_0 }
\end{equation}
where $\alpha$ and $\beta$ are given in \eqref{eq:alphagamma}.
\end{thm}
\begin{proof}
By  \eqref{eq:errKmE} and the Cauchy--Schwarz inequality we have
that
\begin{align*}
\errK{k,p}(t,\veps) & = \norm{\sum\limits_{\ell=0}^{k-1} \eps^\ell ( {c}_\ell(t) -  \widetilde{c}_\ell(t) ) } \\
& \le \sqrt{ \sum\limits_{\ell=0}^{k-1} \abs{\eps^\ell}^2} \
\sqrt{ \sum\limits_{\ell=0}^{k-1} \norm{ {c}_\ell(t) -  \widetilde{c}_\ell(t) }^2} \\
& = \sqrt{ \frac{1-\abs{\eps}^{2k}}{1-\abs{\eps}^2} } \ \norm{\vec(E_{k,p}(t))}.
\end{align*}
The claim follows now from Lemma~\ref{lem:krylov_coefficient_bound}.
%  Theorem 2.1 of \cite{Gallopoulos} and
% the corollary \eqref{eq:log_norm_bound} of Lemma~\eqref{lem:norm_bound}
% we find that
% \begin{equation*} % \label{eq:krylov_error}
% \begin{aligned}
% &\norm{\sum\limits_{\ell=0}^{m-1} \eps^\ell \Big( {c}_\ell(t) -  \widetilde{c}_\ell(t) \Big) } \leq
% \sqrt{ \sum\limits_{\ell=0}^{m-1} \eps^{2\ell} } \norm{\exp(t L_m ) \, \widetilde{u}_0
% - Q_m \exp(t H_m) e_1 \norm{u_0} } \leq \\
% &2 \sqrt{ \frac{1-\eps^{2m}}{1-\eps^2} } \frac{ \norm{tL_m }^m \ee^{\max \{1, \mu(tL_m )\} }}{m!} \norm{u_0} \leq
%  2 \sqrt{ \frac{1-\eps^{2m}}{1-\eps^2} }  \frac{ (t\alpha)^m  \ee^{t\beta}}{m!} \norm{u_0}.
% \end{aligned}
% \end{equation*}
\end{proof}
\subsection{A bound on the truncation error}\label{sect:apriori_taylor}
The previous subsection gives us an estimate for the error in the coefficient vectors $c_\ell(t)$.
To characterize the total error of our approach,
we now analyze the second term in the error splitting \eqref{eq:error_splitting}, i.e., the remainder
\begin{equation*} % \label{def:remainder}
\errT{k}(t,\eps) := \norm{u(t,\eps) - \sum\limits_{\ell=0}^{k-1} \eps^\ell c_\ell(t)}.
\end{equation*}
Lemma \ref{lem:c_bound} of Appendix gives a bound for the norms
of $c_\ell(t)$ and can be used to
derive the following theorem which bounds $\errT{k}(t,\eps)$.

\begin{thm}[Remainder bound] \label{thm:remainder_conv}
Let $c_0$,$c_1$,\dots be the coefficients of the $\varepsilon$-expansion
\eqref{eq:taylor_series}  of $u(t,\varepsilon)$ when
 $N\geq1$. % in \eqref{eq:taylor_series}.
Then, the error $\errT{k}(t,\varepsilon)$ is bounded as
%\norm{ \sum_{\ell=m}^\infty \varepsilon^\ell c_\ell(t)}$ is bounded as
\begin{equation} \label{eq:remainder_bound}
\errT{k}(t,\varepsilon) \leq C_1(t,\varepsilon) \sum\limits_{\ell=0}^{N-1}
\frac{ C_2(t,\varepsilon)^{ \lfloor \frac{k}{N} \rfloor +\ell   } } {(\lfloor \frac{k}{N} \rfloor +\ell -1)!},
\end{equation}
where
\begin{equation} \label{eq:C_1_and_C_2}
\begin{aligned}
C_1(t,\varepsilon) &= \abs{\varepsilon}^{\mathrm{sign}(\abs{\varepsilon} - 1)}
\ee^{t (\mu(A_0) + \ee N a) +  C_2(t,\varepsilon) - 1}   \norm{u_0}, \\
%\quad \textrm{and} \quad
C_2(t,\varepsilon) &= \abs{\varepsilon}^N \ee N t a.
\end{aligned}
\end{equation}
\end{thm}
\begin{proof}
From Lemma~\ref{lem:c_bound} it follows that
\[
\errT{k}(t,\varepsilon) =
\norm{ \sum_{\ell=k}^\infty \varepsilon^\ell c_\ell(t)} \leq
\sum_{\ell=k}^\infty \abs{\varepsilon}^\ell \norm{c_\ell(t)} \leq
\widetilde{C}_1(t,\varepsilon) \sum\limits_{\ell=k}^\infty \abs{\varepsilon}^\ell \frac{( \ee N t a)^{\lceil \frac{\ell}{N} \rceil}}
{(\lceil \frac{\ell}{N} \rceil -1)!},
\]
where $\widetilde{C}_1(t,\varepsilon) = \ee^{t (\mu(A_0) + \ee N a) - 1} \norm{u_0}$.

Setting $\widetilde{k} = k - N \lfloor \frac{k}{N} \rfloor$
and using the bound $c^\ell = (c^N)^{\frac{\ell}{N}} \leq
(c^N)^{\lceil \frac{\ell}{N} \rceil} c^{ \mathrm{sign}(c - 1)}$ for $c>0$, we get
\begin{equation*}
\begin{aligned}
\sum\limits_{\ell=k}^\infty \abs{\varepsilon}^\ell \frac{( \ee N t a)^{\lceil \frac{\ell}{N} \rceil}}
{(\lceil \frac{\ell}{N} \rceil -1)!}
 & \leq
\abs{\varepsilon}^{ \mathrm{sign}(\abs{\varepsilon} - 1)} \sum\limits_{\ell=\widetilde{k}}^\infty
\frac{(\abs{\varepsilon}^N  \ee N t a)^{\lceil \frac{\ell}{N} \rceil}}
{(\lceil \frac{\ell}{N} \rceil -1)!} \\
 & = \abs{\varepsilon}^{\mathrm{sign}(\abs{\varepsilon} - 1)} \sum\limits_{j=0}^{N-1}
\sum\limits_{\ell = \lfloor \frac{k}{N} \rfloor + j}^\infty
\frac{(\abs{\varepsilon}^N  \ee N t a)^{\ell} }{(\ell-1)!}.
\end{aligned}
\end{equation*}
% where $\widetilde{C}_1(\varepsilon,t) = \varepsilon^N \ee N t a$.
% Assuming $m$ is divisible by $N$, we have...
% \[
% \sum\limits_{\ell=m}^\infty \varepsilon^\ell \frac{( \ee N t a)^{\lceil \frac{\ell}{N} \rceil}}
% {(\lceil \frac{\ell}{N} \rceil -1)!}
% = \varepsilon^{\mathrm{sign}(\varepsilon - 1)} \sum\limits_{k=0}^{N-1}
% \sum\limits_{\widetilde{\ell} = \lceil \frac{m}{N} \rceil + k}^\infty
% \frac{C_2(\varepsilon)^{\widetilde{\ell}}}{(\widetilde{\ell}-1)!}.
% \]
Using the inequality~\cite[Lemma\;4.2]{Saad:1992:MATEXP}
\begin{equation} \label{eq:exp_remainder_bound}
\sum\limits_{\ell=k}^\infty \frac{x^\ell}{\ell !} \leq \frac{x^k \ee^x}{k!} \quad \textrm{for} \quad x>0,
\end{equation}
the claim follows.
\end{proof}
% \begin{thm}[bound] \label{lem:total_error}
% The total error \eqref{def:total_error} is bounded as
% \begin{equation} \label{eq:total_error_bound}
% \err_m(t,\eps) \leq \frac{\ee^{t \mu(A_0)}}{m!} \left( 2 \frac{1-\eps^m}{1-\eps} \alpha^m
% \ee^{t\sum_{\ell=1}^N\norm{A_\ell}} + \norm{\eps A_1}^m \right)XXX
% \end{equation}
% where $\alpha$ is given by \eqref{eq:normbound}. TODO: fix coefficieint in \eqref{eq:total_error_bound}
% \end{thm}
% \begin{proof}
% We bound the error as
% \begin{equation} \label{eq:error_splitted}
%  \norm{u(t,\eps) - \sum\limits_{\ell=0}^{m-1} \eps^\ell \widetilde{c}_\ell(t) } \leq
% \norm{u(t,\eps) - \sum\limits_{\ell=0}^{m-1} \eps^\ell {c}_\ell(t) } +
% \norm{\sum\limits_{\ell=0}^{m-1} \eps^\ell \Big( {c}_\ell(t) -  \widetilde{c}_\ell(t) \Big) } .
% \end{equation}
% For the first term on the right-hand side of \eqref{eq:error_splitted} we use \eqref{}.
% For the second term, we use Theorem~\ref{thm:krylov_coefficient_bound}
% %Theorem 4.3 of
% %\cite{Saad} and
% %Lemma~\eqref{lem:norm_bound}
% %and the corollary \eqref{eq:log_norm_bound} to find that
% \end{proof}

We also give a bound for the special case $N=1$ since it is in this case considerably
lower than the one given in Theorem~\ref{thm:remainder_conv}.
\begin{thm}[Remainder bound $N=1$] \label{thm:remainder_N_1}
Let $N=1$. Then the remainder $\errT{k}$ is bounded as
\begin{equation} \label{eq:bound_N_1}
\errT{k}(t,\veps) \leq \frac{ \ee^{t( \mu(A_0) + \abs{\eps}\norm{A_1}) } ( \abs{ \varepsilon} \norm{t A_1})^k }{k!} \norm{u_0}.
\end{equation}
\end{thm}
\begin{proof}
From Lemma~\ref{lem:c_int_representation} and \eqref{eq:c_sum_representation}
we see that $c_\ell(t)$ consists now of one integral term which can be bounded by
Lemma~\ref{lem:c_term_bound} giving
% , $\ell \geq 1$, is now given by
% \[
% c_\ell(t) = \int\limits_0^t \ee^{(t-t_{i_1})A_0} A_{i_1} \int\limits_0^{t_{i_1}} \ee^{(t_{i_1}-t_{i_2})A_0} A_{i_2}
% \dots \int\limits_0^{t_{i_{\ell-1}}} \ee^{(t_{i_{\ell-1}}-t_{i_\ell})A_0} A_{i_\ell} c_0(t_{i_\ell})
% \, \dd t_{i_1} \ldots \dd t_{i_\ell}.
% \]
%By Lemma~\ref{lem:c_term_bound}, we see that
\[
\norm{c_\ell(t)} \leq \frac{\norm{t A_1}^k}{k!} \, \ee^{t \mu(A_0)} \, \norm{u_0}.
\]
Therefore
\[
\errT{k}(t,\veps) \leq \sum_{\ell=k}^\infty \abs{\varepsilon}^\ell \, \norm{c_\ell(t)} \leq \ee^{t \mu(A_0)}
\sum_{\ell=k}^\infty \frac{( \abs{\varepsilon}  \norm{t A_1})^\ell  }{\ell!} \norm{u_0},
\]
% \leq \frac{\ee^{t \mu(A_0)} (\varepsilon \norm{A_1})^m \ee^{\norm{A_1}}}{m!},
and the claim follows from the inequality \eqref{eq:exp_remainder_bound}.
%$\sum_{\ell=m}^\infty x^\ell/\ell! \leq (x^m \ee^m)/m!$.
\end{proof}

\section{An a posteriori error estimate for the Krylov approximation}\label{sect:aposteori}
Although the previous section provides a proof of convergence the final bound
is not very useful to estimate the error. We therefore also propose the
following a posteriori error estimates, which appear to work well in
the simulations in Section~\ref{sect:examples}.

Let $Q_p \exp(H_p) e_1$ be the approximation of $\ee^A b$, $\norm{b}=1$, by $p$ steps of the Arnoldi method.
Then, due to the fact
that our algorithm is equivalent to the standard Arnoldi method,
the following expansion holds \cite{Saad:1992:MATEXP}
%The error produced by the Arnoldi method applied to the product $\ee^A b$ satisfies the expansion (see \cite{Saad})
\begin{equation} \label{eq:expansion}
  \ee^A b - Q_p \exp(H_p) e_1 = h_{p+1,p} \sum\limits_{\ell=1}^{\infty} e_p^T \varphi_\ell (H_p)e_1 \, A^{\ell-1} q_{p+1},
\end{equation}
where $\varphi_\ell(z) = \sum_{j=0}^\infty \tfrac{z^j}{(j+\ell)!}$ and $q_{p+1}$
is the $(p+1)$st basis vector given by the Arnoldi iteration.

We estimate the error of the Arnoldi approximation
of $\exp( t L_k ) \, \widetilde{u}_0$,
i.e., the approximation of the vector $\vec(E_{k,p}(t))$, % (see \eqref{eq:Em}),
by using the norm of the first two terms in  \eqref{eq:expansion}.
This gives us the estimate
\begin{equation}  \label{eq:vec_est}
\begin{aligned}
\vec(E_{k,p}(t)) & \approx h_{p+1,p} \big( e_p^T \varphi_1(t H_p) e_1 \, q_{p+1} +
e_p^T \varphi_2(t H_p) e_1 \, (t L_k)  q_{p+1} \big) \norm{u_0} \\
& := \widetilde{\mathrm{err}}_{k,p}(t).
\end{aligned}
\end{equation}
Then, for the Krylov error $ \errK{k,p}(t,\veps)$ in the total error
\eqref{eq:error_splitting}, we obtain an estimate $\widetilde{\mathrm{err}}_{\mathrm{K},k,p}(t,\veps)$ directly
using \eqref{eq:errKmE}:
\begin{equation} \label{eq:post_estimate}
\begin{aligned}
\errK{k,p}(t,\veps) & = \|E_{k,p}(t)[1,\varepsilon,\ldots,\varepsilon^{k-1}]^T\| \\
&= \|  \big(I_n \otimes [1,\varepsilon,\ldots,\varepsilon^{k-1}]\big)  \vec(E_{k,p}(t)) \| \\
& \approx \| \big(I_n \otimes [1,\varepsilon,\ldots,\varepsilon^{k-1}]\big) \, \widetilde{\mathrm{err}}_{k,p}(t) \|
=: \widetilde{\mathrm{err}}_{\mathrm{K},k,p}(t,\veps),
\end{aligned}
\end{equation}
% using the Cauchy--Schwarz
% inequality (as in proof of~\ref{thm:krylov_error}):
% \begin{equation} \label{eq:post_estimate}
% \begin{aligned}
% \errK{k,p}(t,\veps) % \leq \sum\limits_{\ell=0}^{m-1} \veps^\ell \norm{c_\ell(t) - \widetilde{c}_\ell(t)}
% & \leq \sqrt{\frac{1- \veps^{2k} }{1- \veps^2 }  } \norm{\vec(E_{k,p}(t))} \\ & \approx
%  \sqrt{\frac{1- \veps^{2k} }{1- \veps^2 }  }  \widetilde{\mathrm{err}}_{k,p}(t)
% =: \widetilde{\mathrm{err}}_{\mathrm{K},k,p} (t,\veps),
% \end{aligned}
% \end{equation}
where $k=1+(N-1)p$. Notice that the
scalars $e_p^T \varphi_1(t H_p) e_1$ and $e_p^T \varphi_2(t H_p) e_1$ in \eqref{eq:vec_est}
can be obtained with a small extra cost using the fact that~\cite[Thm.\;2.1]{AlMohy:2011:EXPINT}
\[
\begin{bmatrix} I_p & 0 \end{bmatrix}
\exp\left( \begin{bmatrix} H_p & e_1 & e_1 \\ 0 & 0 & 0 \\ 0 & 1 & 0 \end{bmatrix} \right)
= \begin{bmatrix} \exp(H_p) & \varphi_1(H_p)e_1 + \varphi_2(H_p)e_1 & \varphi_1(H_p) e_1 \end{bmatrix}.
\]
Since the a priori bound given by Theorem~\ref{thm:remainder_conv} is rather pessimistic in
practice, in numerical experiments we only use the Krylov error estimate \eqref{eq:post_estimate} as
a total error estimate when $N\geq 2$.
For $N=1$, we use also the truncation bound given in Theorem~\ref{thm:remainder_N_1}, i.e., the total estimate is then
\begin{align}
\norm{u(t,\veps) - \widetilde{u}_{k,p}(t,\veps)} & \leq \errT{k}(t,\veps) + \errK{k,p}(t,\veps)
\nonumber \\
& \approx \errKt{k,p}(t,\veps) + \frac{ \ee^{t \left( \mu(A_0) + \abs{\eps} \norm{A_1} \right) }
\left( \abs{ \veps } \norm{t A_1} \right)^p }{p!}.
\label{eq:estimate_N_1}
\end{align}

\section{Numerical examples}\label{sect:examples}

The behavior of the algorithm is now illustrated for two test problems: one stemming from spatial discretization
of an advection-diffusion equation and the other one
appearing in the literature \cite{Lietaert:2015:INTERPOLATORY}
corresponding to the discretization of a damped wave equation.

%the effectivity of the a posteriori error estimate (?)
%will also be depicted.

% However, first we discuss an important aspect concerning the implementation of
% the algorithm, namely the scaling of the matrix $L_m$. The effect of the
% scaling will also be illustrated in the first numerical example.

\subsection{Scaling of $L_m$}

%An important issue for the implementation of the algorithm is the scaling of the matrix $L_m$.
%This is the case, e.g., when $N=1$ and the matrix $A_1$ is of very large norm compared to $A_0$.
%%when approximating the product $\exp(L_m(\mathcal{A}_N) ) \widetilde{u}_0$.
%The scaling is found to affect the numerical stability of the Arnoldi iteration when approximating the
%product $\exp(L_m(\mathcal{A}_N) ) \widetilde{u}_0$,
%%coefficient vectors $c_\ell(t)$. %The scaling is based on the following lemma.
%and it is performed as follows.
%
It turns out that the performance of the algorithm can be improved by
performing a transformation which scales the coefficient matrices.
This scaling can be carried out as follows.
Let $A_0$, $A_1$, \dots, $A_N \in \mathbb{C}^{n \times n}$ and $L_m$
be the corresponding block-Toeplitz matrix of the form \eqref{eq:L_N}.
Let $\gamma>0$ and define $\Sigma_m := \textrm{diag}(1,\gamma,\ldots,\gamma^{m-1}) \otimes I_n$.
Then it clearly holds
\begin{equation*}
\begin{aligned}
\widehat{c}(t) =
\exp(t L_m) \, \widetilde{u}_0
&= \Sigma_m \exp(t \Sigma_m^{-1} L_m \Sigma_m) \, \widetilde{u}_0 \\
& = \Sigma_m \exp(t \widehat{L}_m) \, \widetilde{u}_0,
\end{aligned}
\end{equation*}
where $\widehat{L}_m$ is the matrix \eqref{eq:L_N} corresponding to
$A_0, \gamma^{-1} A_1, \ldots, \gamma^{-N} A_N$.

Thus, we see that using this scaling strategy  corresponds to the changes
\begin{equation} \label{eq:scaling}
\epsilon \rightarrow \gamma \epsilon \quad \textrm{and} \quad
A_\ell \rightarrow \gamma^{-\ell} A_\ell %\quad 1 \leq \ell \leq N.
\end{equation}
when performing the Arnoldi approximation of the product
$\exp(t \widehat{L}_m) \, \widetilde{u}_0$. This is also evident from the original
ODE \eqref{eq:IVP}.

The performance of the algorithm appears to improve when
we scale the norms of coefficients $A_\ell$, $1 \leq \ell \leq N$, such that
they are of the order 1 or less.
%This is the case, e.g., when $N=1$ and the matrix $A_1$ is of very large norm compared to $A_0$
To balance the norms, we used the heuristic choice
\begin{equation} \label{eq:heuristic_gamma}
\gamma = \max\limits_{1 \leq \ell \leq N} \ \norm{A_\ell}^{1/\ell}.
\end{equation}
This was found to work well in all of our numerical experiments, giving both good convergence
and a posteriori error estimates.

We note that scaling has also been exploited for
polynomial eigenvalue problems, e.g.,
in \cite{Fan:2004:NORMWISE}.
Our scaling \eqref{eq:heuristic_gamma} can be interpreted
as a slight variation of the scaling proposed in \cite[Thm.~6.1]{Betcke:2008:OPTIMAL}.
%The effect of scaling is illustrated in the first example.
Another related scaling, one for the matrix exponential of an augmented matrix, can be found in
\cite[p.~492]{AlMohy:2011:EXPINT}.

\subsection{Advection-diffusion operator}

Consider
%the spatial discretization of
the 1-d advection-diffusion equation
\begin{equation} \label{eq:advdiff}
\frac{\partial}{\partial t} y(t,x) = a \, \frac{\partial^2}{\partial x^2} y(t,x) + \eps \, \frac{\partial}{\partial x} y(t,x), \quad y(0,x) = y_0(x)
\end{equation}
with Dirichlet boundary conditions on the interval $[0,1]$
and $y_0(x) = 16 \, ((1-x)x)^2$. The
spatial discretization using
central finite differences gives the ordinary differential equation $u' = (A_0 + \eps A_1) \, u$,
$u(0) = u_0 \in \mathbb{R}^n$, where the matrices $A_0$ and $A_1$ are of the form
\[
A_0 =  \frac{a }{(\Delta x)^2}\begin{bmatrix}
    -2 & 1 & & & \\
    1 & -2 & 1 & & \\
      & \ddots & \ddots & \ddots & \\
      & & 1& -2& 1 \\
      & & & 1 & -2 \end{bmatrix}, \,\,
      A_1 =
 \frac{1}{2 \Delta x}\begin{bmatrix}
     & -1 & & & \\
    1 &   & -1 & & \\
      & \ddots & \ddots & \ddots & \\
      & & 1&  & - 1 \\
      & & & 1 &  \end{bmatrix},
\]
where $\Delta x = (n+1)^{-1}$ and $u_0$ is the discretization of $y_0(x)$.
We set $n=200$ and $a=3 \cdot 10^{-4}$, and approximate at $t=0.5$.
 Then, $\norm{t A_0} \approx 95$.
%We set $A_0 = a \Delta_n$ and $A_1 = \nabla_n$, and
%We observe the errors
%\[
% \norm{  u(t,\eps)  -  \tilde{u}_{k,p}(t,\eps) } .
%\]
%$u_0$ is a discretization of the initial value function. %, \, x \in [0,1]$,
%and the coefficients $c_\ell(t)$ are defined as in \eqref{eq:taylor_series}.
%We compute the partial sums of the form $\eqref{eq:partial_sums}$ when
We compute the approximations $\widetilde{u}_{k,p}(t,\eps)$ for
$\eps =   10^{-3}, 1.5 \cdot 10^{-2}$, and $3 \cdot 10^{-2}$.
Then, respectively, $\norm{t \eps A_1} \approx 0.4,6.0$ and $12.0$.
%and we let $p$ vary from 1 to 70.
%We perform the infinite Arnoldi approximation of the product $\exp\left( t L_m(\mathcal{A}_1) \right) u_0$.
%The partial sums are computed by the product $\exp\left( t L_m(\mathcal{A}_1) \right) u_0$,
%where $ L_m(\mathcal{A}_1)$ is defined as in \eqref{def:L_N}, and $\mathcal{A}_1 = (A_0,A_1)$.
%For the numerical evaluation of the product we use the MATLAB code \texttt{expmv} by Al-Mohy
%and Higham \cite{Almohy_Higham}.
Figure~\ref{fig:figure1} shows the 2-norm errors of these
approximations and the corresponding a posteriori error estimates using \eqref{eq:estimate_N_1}.
We observe superlinear convergence for the error and the estimate.

\begin{figure}[h!]
\begin{center}
\includegraphics[scale=1]{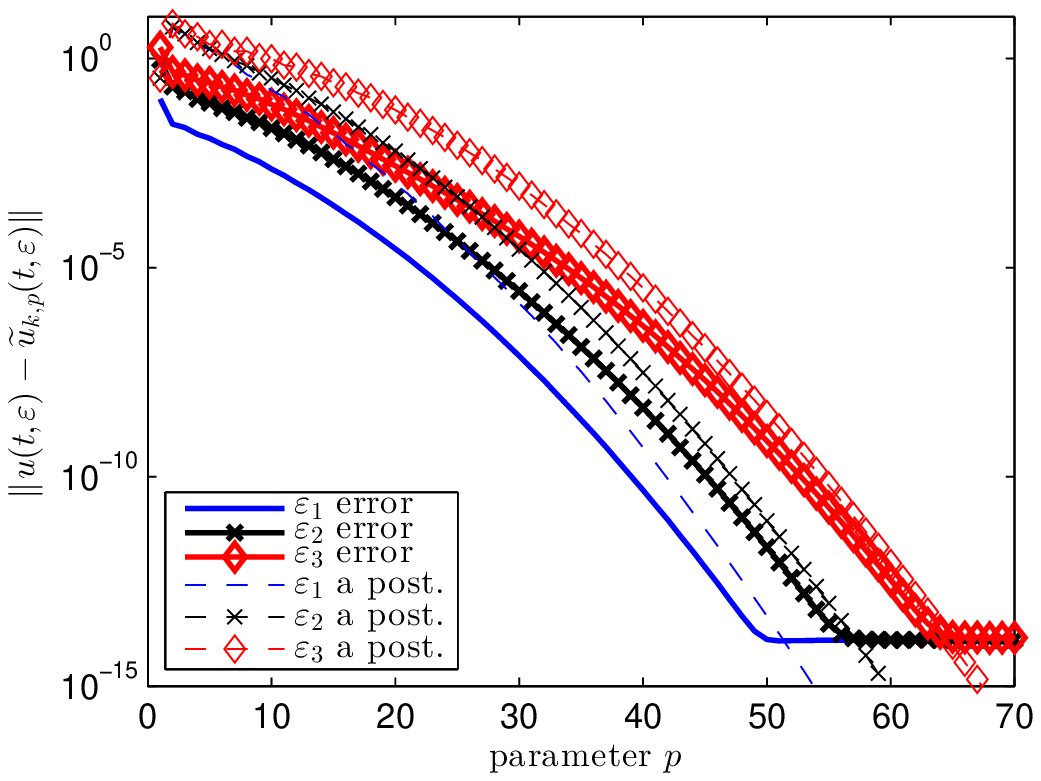}%
\end{center}
\caption{2-norm errors of approximations $\widetilde{u}_{k,p}(t,\eps)$
 and the estimate \eqref{eq:estimate_N_1} when $\eps$ has the values
$\eps_1 = 1\cdot 10^{-3}$, $\eps_2 = 1.5 \cdot 10^{-2}$ and $\eps_3 = 3 \cdot 10^{-2}$.}
\label{fig:figure1}
\end{figure}

To illustrate the generality of our approach we now
consider the case $N=2$, namely a modification of \eqref{eq:advdiff}
\begin{equation} \label{eq:advdiff2}
\frac{\partial}{\partial t} y(t,x) = a \frac{\partial^2}{\partial x^2} y(t,x) + \eps \frac{\partial}{\partial x} y(t,x) + \eps^2 b \, y(t,1-x), \quad y(0,x) = y_0(x);
\end{equation}
the extra term can be interpreted as a non-localized feedback.
We set the parameter $a = 3 \cdot 10^{-4}$ and
$b=2 \cdot 10^2$. The spatial discretization with finite differences
gives the ODE $u' = (A_0 + \eps A_1 + \eps^2 A_2) \, u$,
$u(0) = u_0 \in \mathbb{R}^n$, where $u_0$ and the matrices $A_0$ and $A_1$ are as above, and
\[
A_2 = b \cdot \begin{bmatrix} & & 1 \\ & \Ddots & \\  1 & & \end{bmatrix}.
\]
We compute the approximations $\widetilde{u}_{k,p}(t,\eps)$ for
$\eps =   10^{-3}, 1.5 \cdot 10^{-2}$ and $3 \cdot 10^{-2}$,
for which respectively, $\norm{t \eps^2 A_2} \approx 5.0 \cdot 10^{-4}$, $0.11$, and $0.45$.
%
%and we let $p$ vary from 1 to 70.
Figure~\ref{fig:figure2} shows the  2-norm errors of these
approximations and the corresponding a posteriori error estimates using \eqref{eq:estimate_N_1}.

In Figure~\ref{fig:figure3} we illustrate the dependence of the convergence on the scaling.
Clearly the choice \eqref{eq:heuristic_gamma} results in the fastest convergence for this example.
Simulations with $\gamma$ larger than what is suggested by \eqref{eq:heuristic_gamma} did not
result in substantial improvement of the convergence.

%Figure~\ref{fig:figure2} illustrates the case $N=3$. Matrices $A_0$ and $A_1$ and all
%the parameters are set as above. As matrices $A_2$ and $A_3$ we take random tridiagonal
%matrices, such that for $\eps = 1\cdot 10^{-3}, 1.5 \cdot 10^{-2}, 3 \cdot 10^{-2}$,
%$\veps^2 \norm{A_2} \approx 2 \cdot 10^{-2}, 5 \cdot 10^{-1}, 1.7$ and
% $\veps^3 \norm{A_3} \approx 2 \cdot 10^{-5}, 6 \cdot 10^{-2}, 6 \cdot 10^{-1}.$
\begin{figure}[h!]
\begin{center}
\includegraphics[scale=1]{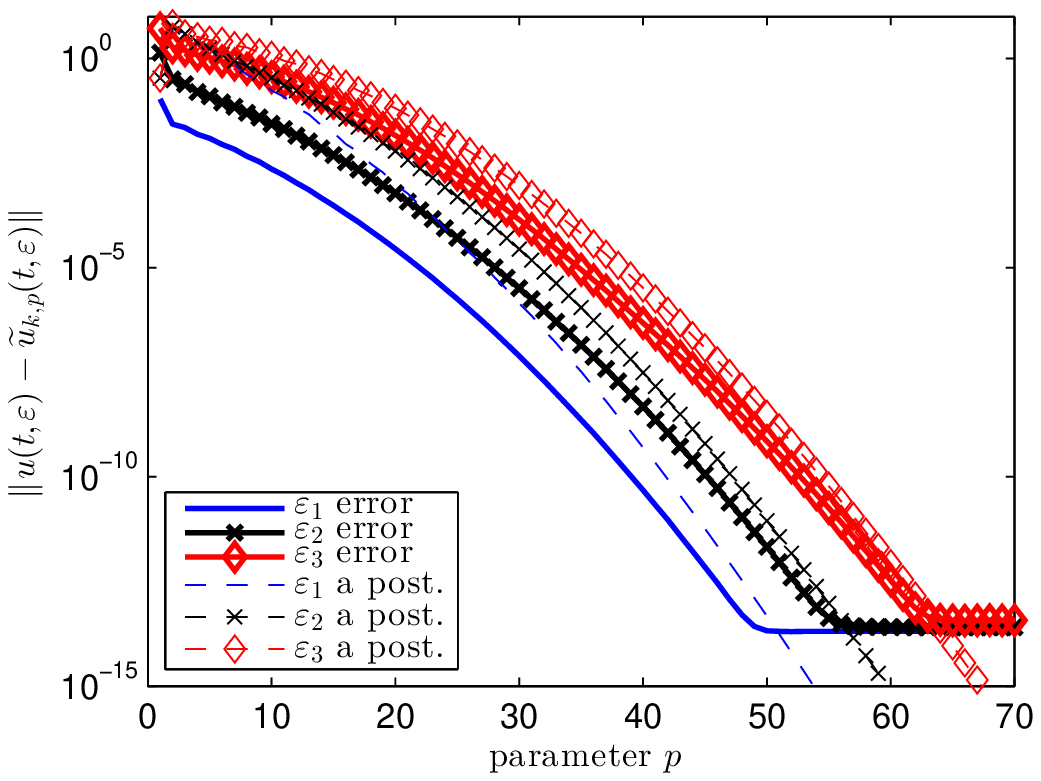}%
\end{center}
\caption{2-norm errors of approximations $\widetilde{u}_{k,p}(t,\eps)$ for the equation \eqref{eq:advdiff2}
 and the error estimate \eqref{eq:post_estimate} when $\eps$ has the values
$\eps_1 = 1\cdot 10^{-3}$, $\eps_2 = 1.5 \cdot 10^{-2}$ and $\eps_3 =  3 \cdot 10^{-2}$.}
\label{fig:figure2}
\end{figure}

\begin{figure}[h!]
\begin{center}
\includegraphics[scale=1]{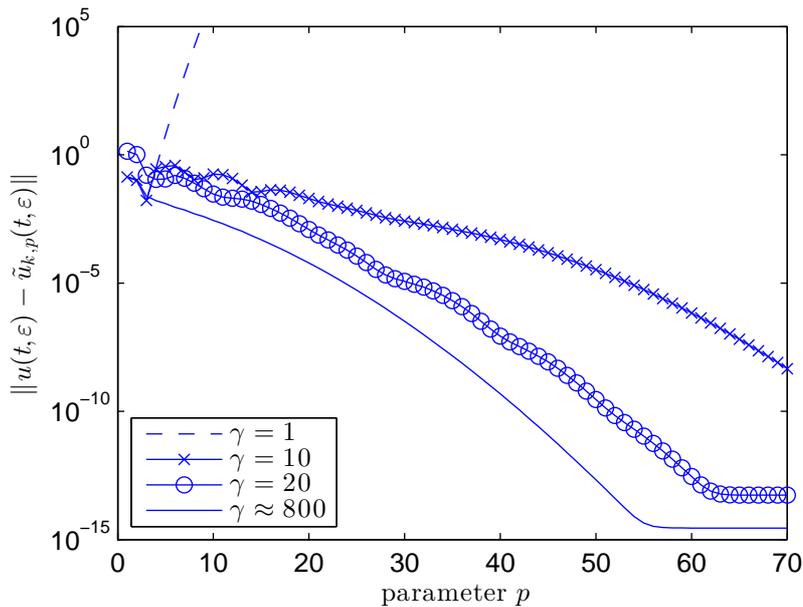}%
\end{center}
\caption{2-norm errors of approximations $\widetilde{u}_{k,p}(t,\eps)$, when $\eps = 1.5 \cdot 10^{-2}$
using different scalings \eqref{eq:scaling}. The last option corresponds to the scaling \eqref{eq:heuristic_gamma}.}
\label{fig:figure3}
\end{figure}

% \begin{figure}[!h]
% \centering
% \begin{varwidth}{2.0\linewidth}  % this is a must
% \hspace{- .0cm}
% %\subfigure[Errors and a posteriori estimates.]
% \subfigure[]
% {\includegraphics[scale=1]{figure2.eps} \label{fig:fig2a}}
% \hspace{-.0cm}
% \subfigure[]
% {\includegraphics[width=7.5cm]{figure3.eps} \label{fig:fig2b} }
% \end{varwidth}
% \caption{(a) shows the error vs.... (b) shows the effect of the scaling on the convergence, where the last option
% equals the scaling \eqref{eq:heuristic_gamma}.}
% \label{fig:figure_array2}
% \end{figure}
% \newpage

\subsection{Wave equation}

Consider next the damped wave equation inside the 3D unit box given in
\cite[Section 5.2]{Lietaert:2015:INTERPOLATORY}.
The governing $2n$-dimensional first-order differential equation is given by
\begin{equation} \label{eq:state_space_ode}
\frac{\dd}{\dd t} \begin{bmatrix} u(t) \\ u'(t) \end{bmatrix}
\begin{bmatrix} 0 & I \\ -M^{-1} K & - M^{-1} C(\gamma) \end{bmatrix}
\begin{bmatrix} u(t) \\ u'(t) \end{bmatrix}, \quad \begin{bmatrix} u(0) \\ u'(0) \end{bmatrix} =
\begin{bmatrix} u_0 \\ u_0' \end{bmatrix} \in \mathbb{R}^{2n},
\end{equation}
where $C(\gamma_1,\gamma_2)  = \gamma_1 C_1 + \gamma_2 C_2.$
The model is obtained by finite differences  with
15 discretization points in each dimension, i.e., $n=15^3$.
The matrix $K$ denotes the discretized Laplacian,
$C(\gamma_1,\gamma_2)$ the damping matrix stemming from Robin boundary conditions, and $M$
the mass matrix. We carry out numerical experiments for parameter values
$\gamma_1 = 0,1,2$ and $\gamma_2 = 0,1,2$.
%$\norm{C_1} = \norm{C_2} \approx 1.1$.
%We approximate at $t=2$, and set as an initial value...

%$\begin{bmatrix} u(0) & u'(0) \end{bmatrix} = \begin{bmatrix} 256((1-x)

We reformulate \eqref{eq:state_space_ode} in the form \eqref{eq:IVP} by setting
\[
A_0 = \begin{bmatrix}  0 & I \\ -M^{-1} K & - M^{-1} \gamma_1 C_1 \end{bmatrix}, \quad
A_1 = \begin{bmatrix} 0 & 0 \\ 0  & - M^{-1} C_2 \end{bmatrix}.
\]
Then, the variable $\veps$ in \eqref{eq:IVP} corresponds to $\gamma_2$.
This means that by running the algorithm for a fixed value
of $\gamma_1$, we may efficiently obtain solutions for different values of $t$ and $\gamma_2$.

Figures~\ref{fig:figure_array} show the contour plots of the numerical solutions  of \eqref{eq:state_space_ode}  at $t=9$
on the plane $\{ (x,y,z) \in [0,1]^3 \, : \, z = 0.5\}$ for
different values of $(\gamma_1,\gamma_2)$.
Note that for a fixed value of $\gamma_1$, only one run of the
algorithm is required to compute the solution for many different $\gamma_2$.
%0,20,40$, $\gamma_2 = 0,20,40.$
%different choices of $\gamma_1$ and $\gamma_2$.

In Figure~\ref{fig:figure4} we illustrate the relative 2-norm errors of the approximations,
when $\gamma_1 = 2$ and $\gamma_2 = 0,1$ and $2$. Then, $\norm{t A_0} \approx 108$, and, respectively,
$\norm{t \gamma_2 A_1} \approx 0, 9.6$ and $12.9$.
We again observe superlinear convergence, and, moreover,
the a posteriori error estimate is very accurate for this example.
 % and $\gamma_2 = 0.2,1.8,3.4,5.0$.

\begin{figure}[!h]
\centering
\begin{varwidth}{2.0\linewidth}  % this is a must
%\hspace{-10mm}
\subfigure[$\gamma_1=\gamma_2=0$]{\includegraphics[width=4.25cm]{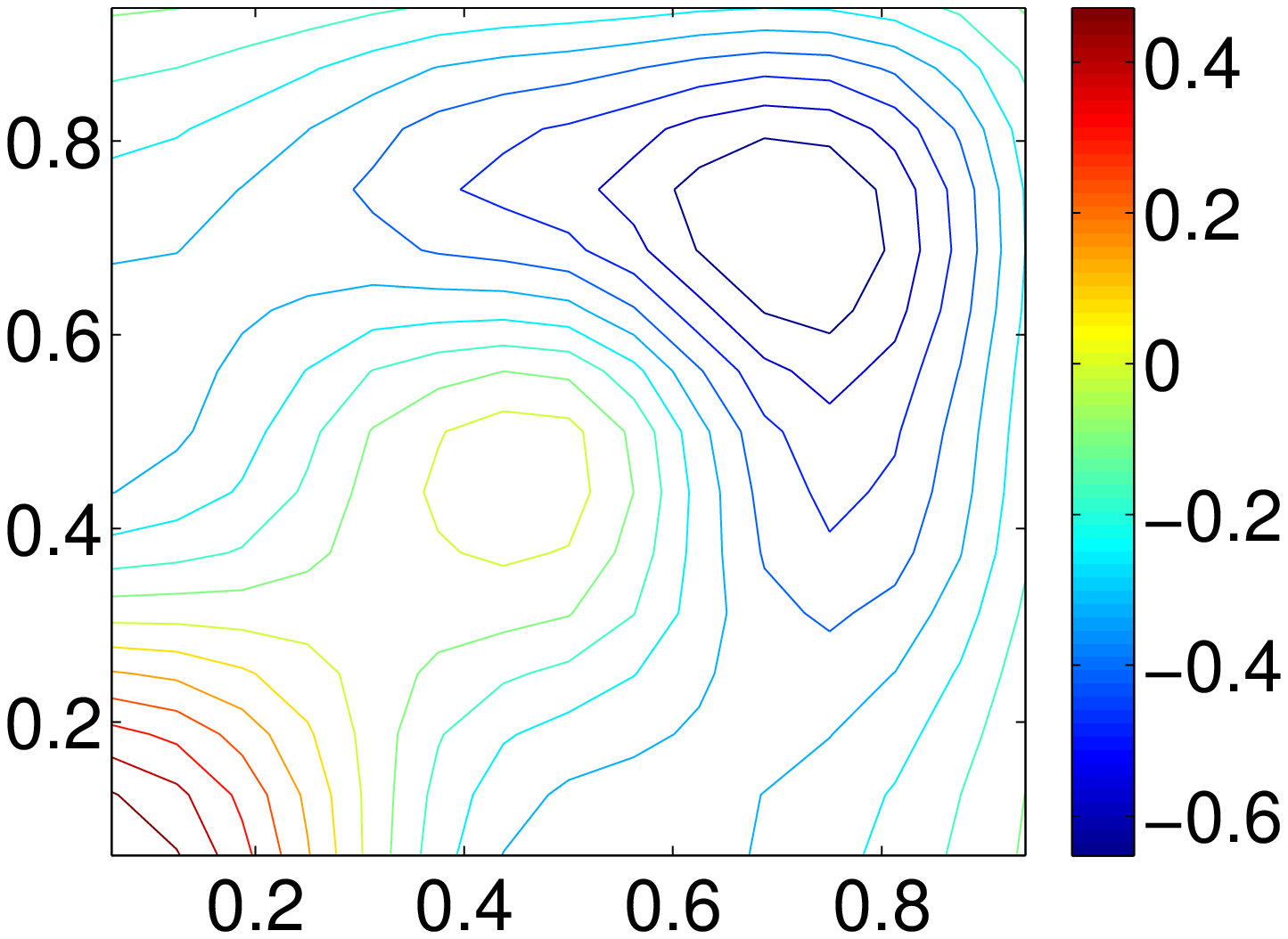}}
\subfigure[$\gamma_1=1$, $\gamma_2= 0$]{\includegraphics[width=4.25cm]{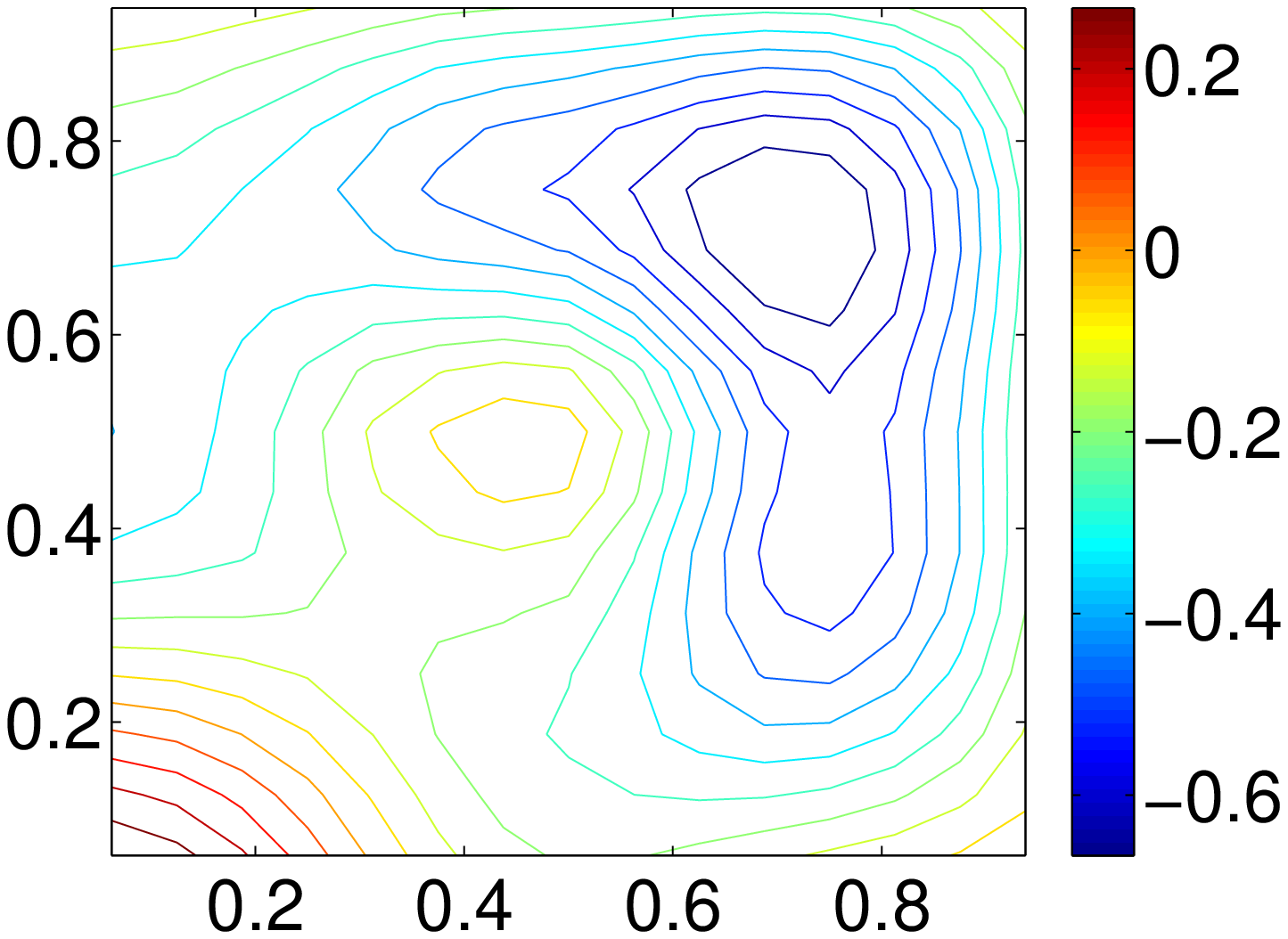}}
\subfigure[$\gamma_1 = 2$, $\gamma_2 =  0$]{\includegraphics[width=4.25cm]{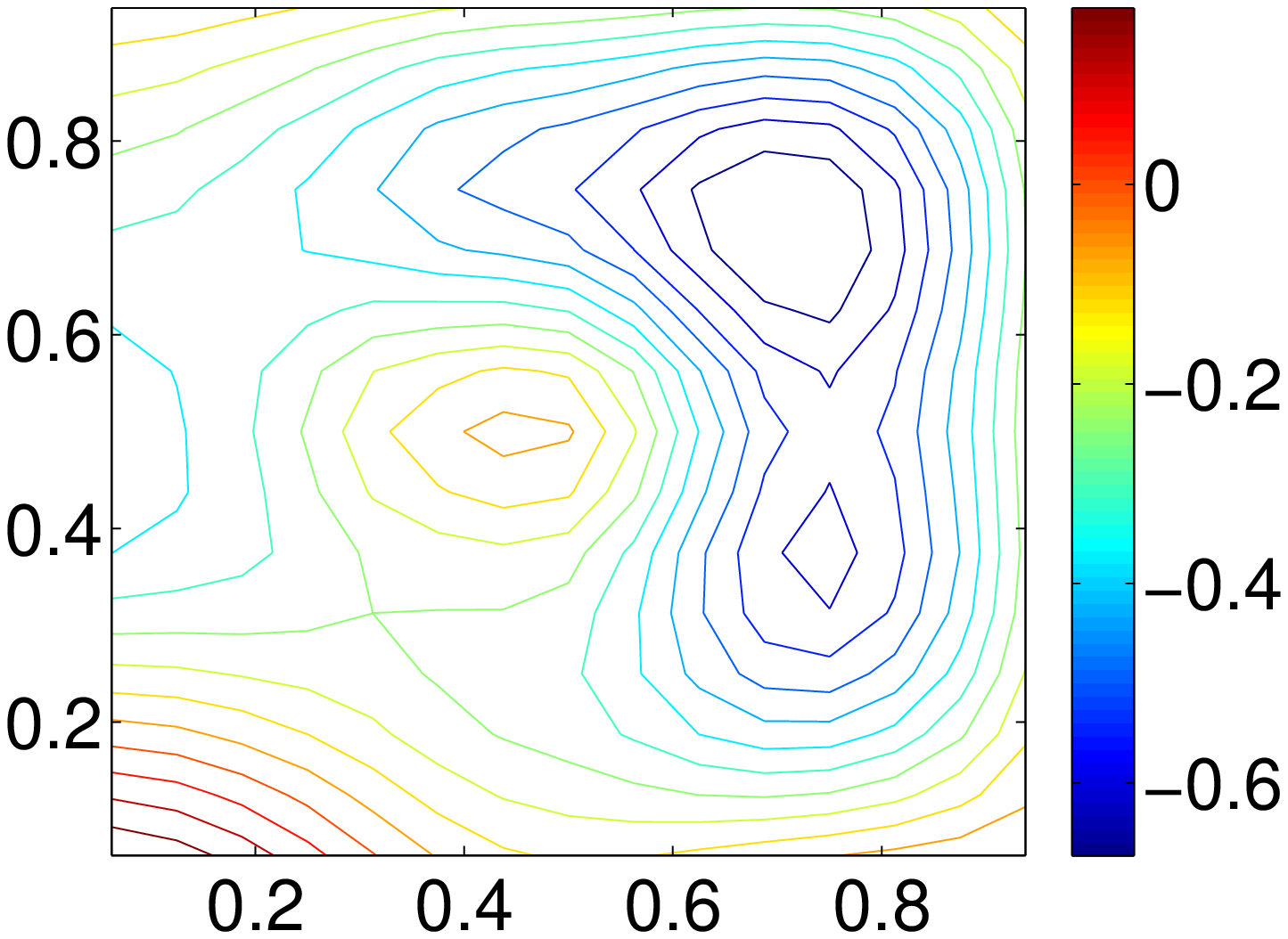}} \\
%\end{varwidth}
%\begin{varwidth}{0.5\linewidth}  % this is a must
\subfigure[$\gamma_1=1$, $\gamma_2=1$]{\includegraphics[width=4.25cm]{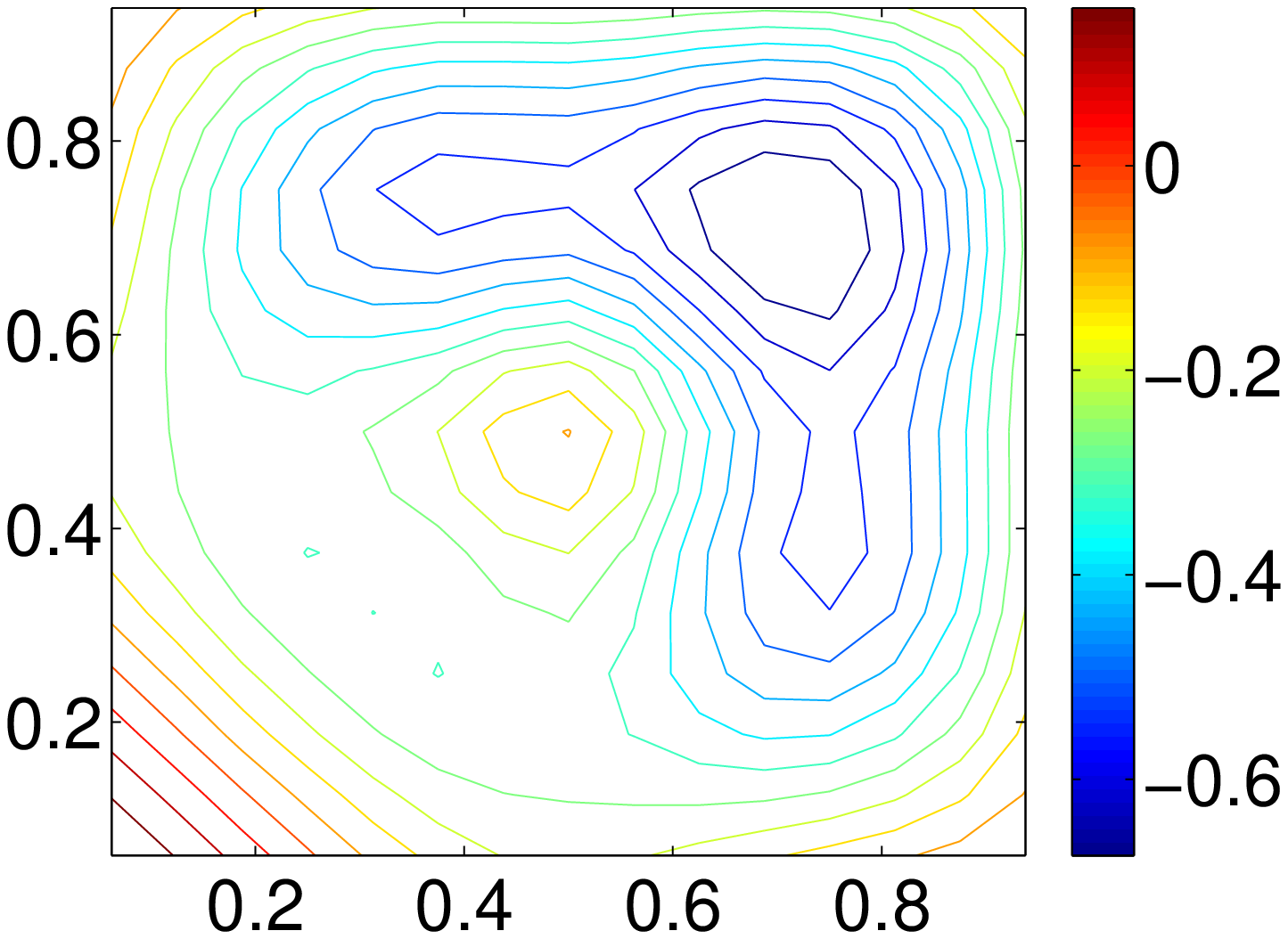}}
\subfigure[$\gamma_1= 1$, $\gamma_2=2$]{\includegraphics[width=4.25cm]{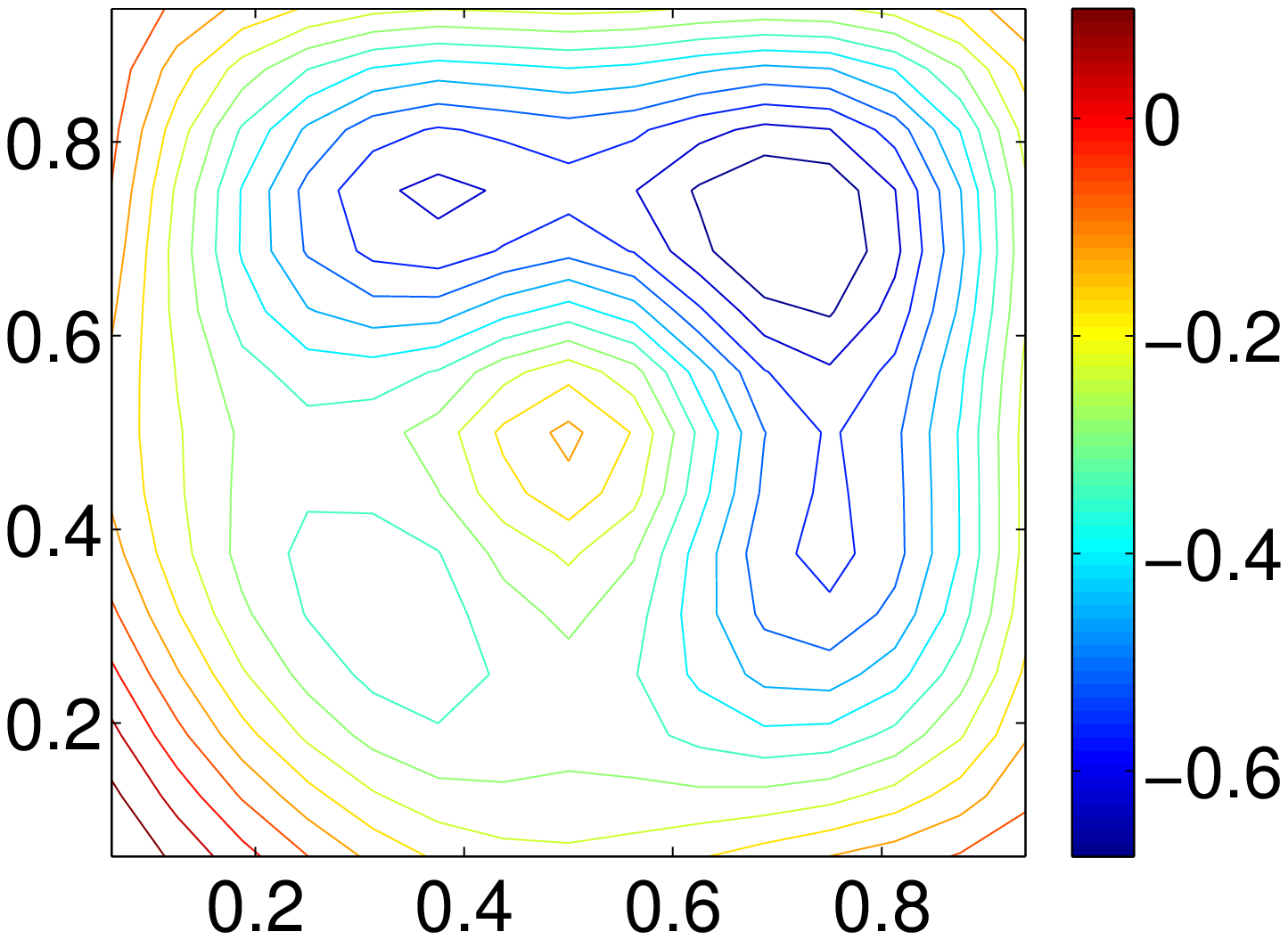}}
\subfigure[$\gamma_1=2$, $\gamma_2=2$]{\includegraphics[width=4.25cm]{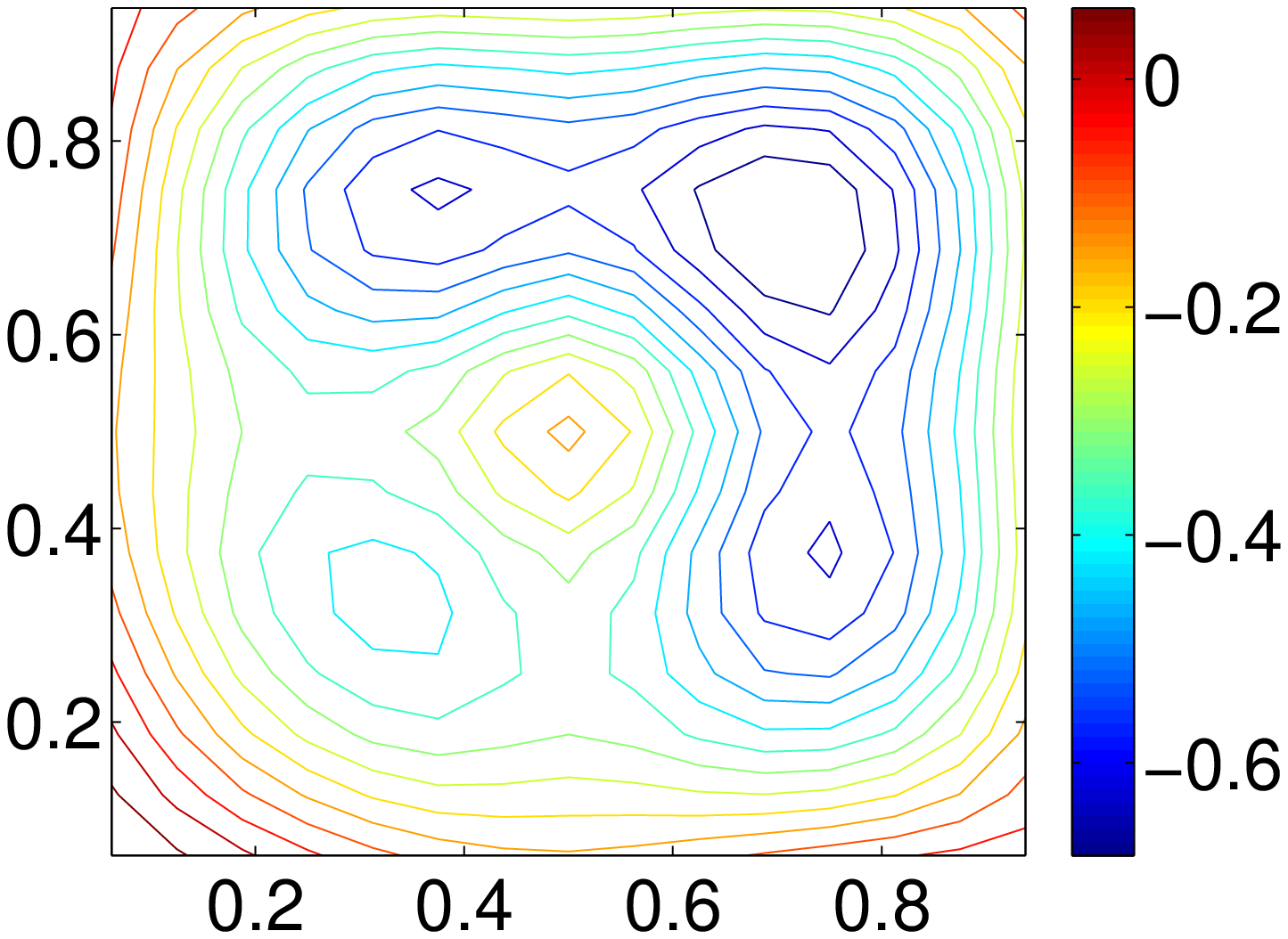}}
\end{varwidth}

\caption{The solution to \eqref{eq:state_space_ode}
in the plane $z=0.5$, for different values of $(\gamma_1,\gamma_2)$ at $t=9$.}
\label{fig:figure_array}
\end{figure}

\begin{figure}[h!]
\begin{center}
\includegraphics[scale=1]{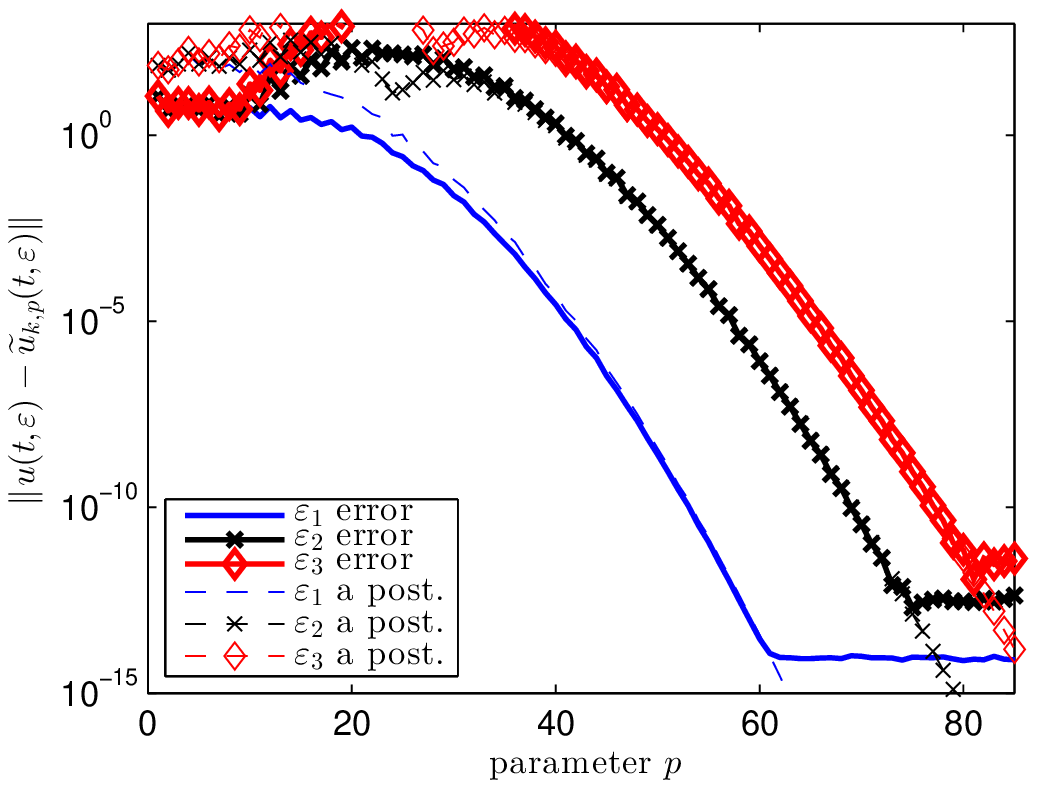}%
\end{center}
%{
%2-norm errors of the partial approximations
% and the estimate \eqref{eq:estimate_N_1}, when
%$\eps = 1\cdot 10^{-3}, 1.5 \cdot 10^{-2}, 3 \cdot 10^{-2}$.
%}
\caption{2-norm errors of approximations $\widetilde{u}_{k,p}(t,\eps)$ for the equation \eqref{eq:advdiff2}
 and the error estimate \eqref{eq:post_estimate}, when $\gamma_1=2$ and $\gamma_2$ has the values
$\veps_1 = 1$, $\veps_2 = 1.5$ and $\veps_3 =  2$.
}
\label{fig:figure4}
\end{figure}

\section{Conclusions and outlook}

The focus of this paper is an algorithm
for parameterized linear ODEs, which is shown to have
superlinear convergence in theory and perform
convincingly in several examples. The behavior
is consistent with what is expected from an Arnoldi method.
Due to the equivalence with the Arnoldi method, the algorithm
may suffer from the typical disadvantages of the Arnoldi method,
for instance, the fact that the computation time per iteration increases with
the iteration number. The standard approach to resolve this issue
is by using restarting, which we leave for future work.

\bibliographystyle{plain}
\bibliography{eliasbib}
%TODO at a late stage: merge reference lists

\appendix

\section{Technical lemmas for the norm and the field of values of $L_m$}

We now provide bounds on the norm and the field of values of $L_m$,
which are needed in Section~\ref{sect:apriori_krylov}.
%These bounds are needed to bound the error produced by the Krylov approximation
%using the results of \cite{Saad}.
The derivation is done with properties of field of values.
Recall that the field of values of  a matrix $A \in \mathbb{C}^{n \times n}$
is defined as
\[
\mathcal{F}(A) = \{ x^* A x \, : \, x \in \mathbb{C}^n, \ \norm{x} = 1 \}.
\]

%Let $L_m $ be defined as in \eqref{eq:L_N}.
The bounds for the norm and the field of values of $\mathcal{A}_N$
follow from the block structure of $L_m $.
\begin{lem} \label{lem:norm_bound}
Let $N\geq 0$ and $L_m $ be given by \eqref{eq:L_N}. Then,
\[
\norm{L_m } \leq \sum\limits_{\ell=0}^N \norm{A_\ell}
\]
\end{lem}
\begin{proof}
Let $x = [x_1^\mathrm{T} \ \ldots \ x_m^\mathrm{T}]^\mathrm{T} \in \mathbb{C}^{nm}$ such that
$x_i \in \mathbb{C}^n$ for all $1\leq i \leq m$ and $\norm{x} = 1$. From the block Toeplitz structure
of $L_m$ we see that
\[
\norm{L_m x} \leq \sum\limits_{\ell=0}^N \sqrt{ \sum\limits_{k=1}^{n-\ell} \norm{A_\ell x_k}^2 }
\leq \sum\limits_{\ell=0}^N \norm{A_\ell} \sqrt{ \sum\limits_{k=1}^{n} \norm{x_k}^2 } = \sum\limits_{\ell=0}^N \norm{A_\ell}.
\]
\end{proof}

Next, we give a bound for the field of values of the matrix $L_m $.
Let $d(\mathcal{S}, z)$ denote the distance between a closed set $\mathcal{S}$ and a point $z$.
%The field of values $\mathcal{F}(L_m )$ can be bounded as follows.
\begin{lem} \label{lem:fov}
Let $N\geq 0$ and $L_m$ be given by \eqref{eq:L_N}. Then,
\[
\mathcal{F}(L_m ) \subset \{ z \in \mathbb{C} \, : \,
d(\mathcal{F}(A_0),z) \leq  \sum\limits_{\ell=1}^N \norm{A_\ell}  \}.
\]
\end{lem}
\begin{proof}
Let $x = [x_1^\mathrm{T} \ \ldots \ x_m^\mathrm{T}]^\mathrm{T} \in \mathbb{C}^{nm}$,
where $x_i \in \mathbb{C}^n$ for all $1\leq i \leq m$ and $\norm{x} = 1$. Then
\begin{equation} \label{eq:splitted}
x^* L_m  x = \sum\limits_{\ell = 1}^m x_\ell^* A_0 x_\ell +
% \sum\limits_{k=1}^N \sum\limits_{\ell=1}^{m-k+1} x_{\ell+k}^* A_k x_\ell.
x^* \widetilde{L}_m  x,
\end{equation}
where $\widetilde{L}_m $ equals $L_m $ on the subdiagonal blocks and
is otherwise zero.
By the convexity of the field of values \cite[Property~1.2.2]{Horn:1991:MATAN2}, we know that
$\sum_{\ell = 1}^m x_\ell^* A_0 x_\ell \in \mathcal{F}(A_0)$.
The second term in \eqref{eq:splitted} can be bounded as in proof of Lemma~\ref{lem:norm_bound},
giving
\[
|x^* \widetilde{L}_m  x| \leq \norm{\widetilde{L}_m }
\leq \sum\limits_{\ell=1}^N \norm{A_\ell}.
\]
\end{proof}

As a corollary of Lemma~\ref{lem:fov}, we have
 the following bound, which follows directly from the fact that the logarithmic norm of a matrix
in 2-norm equals the real part of the rightmost point in its field of values.
\begin{cor} \label{cor:log_norm_bound}
Let $L_m $ be given by \eqref{eq:L_N}. Then,
\begin{equation} \label{eq:log_norm_bound}
\mu(L_m) \leq \mu(A_0) + \sum\limits_{\ell=1}^N \norm{A_\ell},
\end{equation}
where $\mu(A)$ denotes the logarithmic norm of $A$ defined
in \eqref{eq:lognorm}.
%, i.e.,
%$\mu(A) = \max \{ \lambda \, : \, \lambda \in \Lambda(\frac{A+A^*}{2}) \}$.
\end{cor}

\section{Technical lemmas for coefficient bounds}
We now derive bounds needed for the a priori
analysis of the truncation error in Section~\ref{sect:apriori_taylor}.
The following result provides an explicit characterization of the expansion
 coefficients. The proof technique is based on the same
type of reasoning as what is commonly
 used in the analysis of Magnus series expansions
for time-dependent ODEs; see, e.g., \cite{Blanes:2009:MAGNUS}.
\begin{lem}[Explicit integral form] \label{lem:c_int_representation}
Let $\ell$ and $N$ be positive integers such that $N\leq \ell$.
Denote by $C_\ell$ the set of compositions of $\ell$, i.e.,
\begin{equation}\label{eq:Cell_def}
C_\ell = \{ (i_1,\ldots,i_r) \in \NN_+^r \, : \, i_1 + \cdots + i_r = \ell \},
\end{equation}
and further denote
\begin{equation} \label{eq:C_m_N}
C_{\ell,N} := \{ (i_1,\ldots,i_r) \in C_\ell \, : \, i_s \leq N \,\,\, \textrm{for all} \,\,\,  1\leq s \leq r \}.
\end{equation}
Then,
\begin{equation} \label{eq:c_sum_representation}
\begin{aligned}
c_0(t) & = \ee^{t A_0} u_0, \\
c_\ell(t) & =  \sum\limits_{(i_1,\ldots,i_r) \in C_{\ell,N}}
 \int\limits_0^t \ee^{(t-t_{i_1})A_0} A_{i_1} \int\limits_0^{t_{i_1}} \ee^{(t_{i_1}-t_{i_2})A_0} A_{i_2}\\
 & \qquad \dots \int\limits_0^{t_{i_{r-1}}} \ee^{(t_{i_{r-1}}-t_{i_r})A_0} A_{i_r} c_0(t_{i_r}) \, \dd t_{i_1} \ldots \dd t_{i_r}
\quad \textrm{for} \quad \ell>0.
\end{aligned}
\end{equation}
\end{lem}
\begin{proof}
%First note that Lemma~\ref{lem:block_toeplitz}
%is applicable to any $m$.
%and \eqref{eq:c_de}
%we see the vectors $c_\ell(t)$ satisfy
From the ODE \eqref{eq:c_ode} and the variation-of-constants formula it follows that
%\begin{equation} \label{eq:c_de2_b}
%\begin{aligned}
%c_0'(t) & = A_0 c_0(t), \quad c_0(0) = u_0, \\
% c_\ell'(t) &= A_0 c_\ell(t) + \sum_{k=1}^{\min(N,\ell)} A_k c_{\ell-k}(t), \quad c_\ell(0) = 0
%\quad \textrm{for} \quad  \ell >0.
%\end{aligned}
%\end{equation}
%By the variation-of-constants formula the solution of this system of ODEs is given by
%We first show the following statement:  The vectors $c_\ell(t)$ satisfy the recursion
\begin{subequations}\label{eq:c_l_recursion}
\begin{eqnarray}
  c_0(t) &=& \ee^{tA_0}u_0\label{eq:c_l_recursion_0},\\
  c_\ell(t) &=&\sum\limits_{k=1}^{\min\{N,\ell\}} \int\limits_0^t \ee^{(t-s)A_0} A_k c_{\ell-k}(s) \, \dd s
\qquad \textrm{for}  \quad \ell > 0.\label{eq:c_l_recursion_ell}
\end{eqnarray}
\end{subequations}
%\begin{subequations}\label{eq:c_de2}
%\begin{eqnarray}
% c_0'(t)  &=& A_0 c_0(t), \quad c_0(0) = u_0\label{eq:c_de2_a}\\
% c_\ell'(t)  &=& A_0 c_\ell(t) + \sum_{k=1}^{\min(N,\ell)} A_k c_{\ell-k}(t), \quad c_\ell(0) = 0 \quad \textrm{for} \quad \ell > 0.\label{eq:c_de2_b}
%\end{eqnarray}
%\end{subequations}
% This can be proven by induction and using Lemma~\ref{lem:block_toeplitz}
% several times. We first apply Lemma~\ref{lem:block_toeplitz}
% with $m=1$, which directly gives \eqref{eq:c_l_recursion_0}.
%  Suppose \eqref{eq:c_l_recursion_ell} is
% satisfied for some value $\ell-1$. We apply
% Lemma~\ref{lem:block_toeplitz} with $m=\ell+1\ge 1$. By differentiating
% \eqref{eq:block_toeplitz} with respect to $t$, the last
% row of \eqref{eq:block_toeplitz} becomes
% \begin{equation}
%  c_\ell'(t)  &=& A_0 c_\ell(t) + \sum_{k=1}^{\min(N,\ell)} A_k c_{\ell-k}(t), \quad c_\ell(0) = 0.\label{eq:c_de2_b}
% \end{equation}
% Equation \eqref{eq:c_l_recursion_ell} follows by
% applying the variation-of-constants formula to \eqref{eq:c_de2_b}, which
% completes the induction for \eqref{eq:c_l_recursion}.
Using \eqref{eq:c_l_recursion} we now prove \eqref{eq:c_sum_representation}
by induction. For $\ell=1$, we have $C_1=\{(1)\}$ and $C_{1,1}=\{(1)\}$.
From \eqref{eq:c_l_recursion_0} and \eqref{eq:c_l_recursion_ell}
we directly conclude that
\[
  c_1(t)=\int_{0}^t\ee^{(t-t_1)A_0}A_1c_0(t_1)\,\dd t_1.
\]
Suppose \eqref{eq:c_sum_representation} holds
for  $\ell=1,\ldots,p-1$ for some value $p>1$. From Definition~\eqref{eq:Cell_def},
we know that the row sum of any element of $C_{p}$ is $p$,
and the row sum of any element in $C_{p-k}$ is $p-k$.
Therefore,  $C_p$ satisfies the recurrence relation
\begin{equation}\label{eq:Cell}
  C_{p}=\bigcup_{k=1}^p \;
\bigcup_{(i_1,\ldots,i_r)\in C_{p-k}}\;\;(k,i_1,\ldots,i_r)
\end{equation}
and
\begin{equation}\label{eq:CellN}
  C_{p,N}=\bigcup_{k=1}^{\min(N,p)} \;
\bigcup_{(i_1,\ldots,i_r)\in C_{p-k,N}}\;\;(k,i_1,\ldots,i_r).
\end{equation}

%\begin{equation}\label{eq:Cell}
%  C_{\ell}=\bigcup_{(i_1,\ldots,i_r)\in C_{\ell-1}}
%\,\bigcup_{j=1,\ldots r+1}c+e_j.
%\end{equation}
%In \eqref{eq:Cell}, $e_j$ denotes the $j$th unit vector and
%we have interpreted addition of vectors
%of different length by padding the shorter vector with zero elements.
%
%
By using \eqref{eq:c_l_recursion_ell} with $\ell=p$ and the fact that
\eqref{eq:c_sum_representation} is assumed to be satisfied for $\ell=1,\ldots,p-1$ we have
\begin{equation*}
\begin{aligned}
  c_p(t) =  &
\sum\limits_{k=1}^{\min\{N,p\}} \int\limits_0^t \ee^{(t-s)A_0} A_k c_{p-k}(s) \, \dd s \\
&= \sum\limits_{k=1}^{\min\{N,p\}} \int\limits_0^t \ee^{(t-s)A_0} A_k
\sum\limits_{(i_1,\ldots,i_r) \in C_{p-k,N}}
 \int\limits_0^s \ee^{(s-t_{i_1})A_0} A_{i_1} \dots \\
& \qquad \int\limits_0^{t_{i_{r-1}}} \ee^{(t_{i_{r-1}}-t_{i_r})A_0} A_{i_r} c_0(t_{i_r})\dd t_{i_1} \ldots \dd t_{i_r}\dd s.
\end{aligned}
\end{equation*}
By rearranging the terms as a double sum and using \eqref{eq:CellN}, we have
\begin{equation*}
\begin{aligned}
c_p(t) &= \sum\limits_{k=1}^{\min\{N,p\}}\sum\limits_{(i_1,\ldots,i_r) \in C_{p-k,N}}
\int\limits_0^t \ee^{(t-s)A_0} A_k
 \int\limits_0^s \ee^{(s-t_{i_1})A_0} A_{i_1} \dots \\
& \quad \quad \int\limits_0^{t_{i_{r-1}}}
\ee^{(t_{i_{r-1}}-t_{i_r})A_0} A_{i_r} c_0(t_{i_r})\dd t_{i_1} \ldots \dd t_{i_r}\dd s \\
 &=
\sum\limits_{(i_1,\ldots,i_r) \in C_{p,N}}
 \int\limits_0^t \ee^{(t-t_{i_1})A_0} A_{i_1} \dots \int\limits_0^{t_{i_{r-1}}}
\ee^{(t_{i_{r-1}}-t_{i_r})A_0} A_{i_r} c_0(t_{i_r})\dd t_{i_1} \ldots \dd t_{i_r}
\end{aligned}
\end{equation*}
which shows that \eqref{eq:c_sum_representation} holds for $\ell=p$ and completes the proof.
%The conclusion \eqref{eq:c_sum_representation} for $\ell=p$
% follows from the fact that  the double
%sum  is the same as a single sum over $C_{p,N}$
%due to formula \eqref{eq:CellN}.
%By the variation-of-constants formula, the solution of \eqref{eq:c_de2} is given by \eqref{eq:c_l_recursion}.
%The claim follows by induction from Lemma~\ref{lem:c_l_recursion} and the representation \eqref{eq:c_l_recursion}.

\end{proof}

\begin{lem} \label{lem:c_term_bound}
Let $m,N$ be positive integers, $N\leq m$, and let $C_{m,N}$ be defined as in \eqref{eq:C_m_N}.
Let $(i_1,i_2,\ldots,i_r) \in C_{m,N}$,
$a = \max_{j = 1, \dots, N} \norm{A_j}$,
and assume that for all $1 \leq j \leq N $.
Then, one corresponding term in \eqref{eq:c_sum_representation} is bounded as
\begin{equation} \label{eq:int_bound}
\begin{aligned}
& \quad \norm{
 \int\limits_0^t \ee^{(t-t_{i_1})A_0} A_{i_1} %\int\limits_0^{t_{i_1}} \ee^{(t_{i_1}-t_{i_2})A_0} A_{i_2}
\ldots \int\limits_0^{t_{i_{r-1}}} \ee^{(t_{i_{r-1}}-t_{i_r})A_0} A_{i_r} c_0(t_{i_r}) \, \dd t_{i_1} \ldots \dd t_{i_r}
} \\
& \qquad \leq  \ee^{t \mu(A_0)} \frac{(ta)^r}{r!} \norm{u_0}.
\end{aligned}
\end{equation}
\end{lem}
\begin{proof}
By using the Dahlquist bound \cite[p.~138]{Trefethen:2005:PSEUDOSPECTRA}
for the matrix exponential, the rightmost integral on the left-hand
side of \eqref{eq:int_bound} can be bounded as
\begin{equation*}
\begin{aligned}
&  \norm{\int\limits_0^{t_{i_{r-1}} }\ee^{(t_{i_{r-1}}-t_{i_r}) A_0} A_{i_r} \ee^{t_{i_r} A_0} u_0 \, \dd t_{i_r}} \\
\leq &\int\limits_0^{t_{i_{r-1}}}  \ee^{(t_{i_{r-1}}-t_{i_r}) \mu(A_0)} \norm{A_{i_r}} \ee^{{t_{i_r}} \mu(A_0)} \dd t_{i_r} \norm{u_0}
=   t_{i_{r-1}} a \ee^{t_{i_{r-1}} \mu(A_0)} \, \norm{u_0}.
\end{aligned}
\end{equation*}
The claim \eqref{eq:int_bound} follows
by applying the same bounding technique $r-1$ times for the remaining integrals,
and using that $t_i\le t$ for any $i$.
%and noting that $t_i\le r$ for all $i$.
%Starting from the rightmost integral and bounding similarly one integral at a time, the claim follows.
\end{proof}

%\begin{lem} \label{lem:conv_N_1}
%Consider the $\epsilon$-expansion for $\mathcal{A}_1 = (A_0, A_1)$. The remainder $R_m(\eps)$ is bounded as follows:
%\begin{equation} \label{eq:estimate1}
%\norm{R_m(\eps)} \leq \frac{\norm{t A_1}^{m+1} \eps^{m+1}}{(m+1)!}\ee^{t \omega( A_0)} .
%\end{equation}
% \end{lem}
%
\begin{lem}[Coefficient bound] \label{lem:c_bound}
Let $c_0$,$c_1$,\dots be  the $\varepsilon$-expansion of $u(t,\varepsilon)$
in \eqref{eq:taylor_series}
for $N\geq1$ in \eqref{eq:taylor_series},
and let $a = \max_{j = 1, \dots, N} \norm{A_j}$.
Then, for any $\ell\ge 0$ such that $k:=\lceil \frac{\ell}{N} \rceil\ge 2$,
%the coefficients $c_\ell(t)$ of the expansion are bounded by
\[
\norm{c_\ell(t)} \leq
\ee^{t (\mu(A_0) + \ee N a)-1}\frac{(\ee   N t a)^k }{(k-1)!} \norm{u_0}.
\]
\end{lem}
\begin{proof}
We first note that the maximum length of any
vector in $C_{m,N}$ is $m$, and
the vector with the shortest length has length at least $k=\lceil \frac{m}{N} \rceil$.
%By simple combinatorics we see from the
%representation of Lemma~\ref{lem:c_int_representation} that
Hence, Lemma~\ref{lem:c_int_representation} can
be rephrased as
\begin{equation}\label{eq:c_separated_sum}
\begin{aligned}
c_m(t) & = \sum\limits_{\widehat{r}=k}^m \ \underset{r=\widehat{r}}{\sum\limits_{ (i_1,\ldots,i_r) \in C_{m,N} }}
\int\limits_0^t \ee^{(t-t_{i_1})A_0} A_{i_1} \\
 & \qquad \dots \int\limits_0^{t_{i_{r-1}}} \ee^{(t_{i_{r-1}}-t_{i_r})A_0} A_{i_r} c_0(t) \, \dd t_{i_1} \ldots \dd t_{i_r}.
\end{aligned}
\end{equation}
Since, $C_{m,N}\subset C_m$ we can bound
the number of elements in $C_{m,N}$
\[
\# \{ (i_1,\ldots,i_r) \in C_{m,N} \, : \, r=\ell \} \leq \# \{ (i_1,\ldots,i_r) \in C_m \, : \, r=\ell \} = \binom{m-1}{\ell-1},
\]
and Lemma~\ref{lem:c_term_bound} and
\eqref{eq:c_separated_sum} imply that
\[
\norm{c_m(t)} \leq \ee^{t \mu(A_0)} \sum\limits_{\ell=k}^m \binom{m-1}{\ell-1} \frac{(ta)^\ell}{\ell !} \norm{u_0}.
\]
Moreover,
\begin{equation*}
 \begin{aligned}
 \binom{m-1}{r-1} & = \frac{(m-1)(m-2)\cdots (m-r+1)}{(r-1)!}  \leq \frac{ m^r }{(r-1)!}
 \end{aligned}
\end{equation*}
and therefore
\begin{equation}\label{eq:cm_norm_proof}
\norm{c_m(t)} \leq \ee^{t \omega(A_0)} \sum\limits_{r=k}^m \frac{ m^r (ta)^r}{ (r-1)! \, r!} \norm{u_0} \leq
\ee^{t \mu(A_0)} \sum\limits_{r=k}^m \frac{ r^r (Nta)^r}{ (r-1)! \, r!} \norm{u_0}.
\end{equation}
In the second inequality in \eqref{eq:cm_norm_proof}
we use $m=N \frac{m}{N}\le N \lceil \frac{m}{N} \rceil=N k\le Nr$.
Using the inequality $\ee \left( \frac{n}{\ee}  \right)^n \leq n!$, $n\ge 1$, we see that for $k\ge 2$
%Using the Stirling's formula $r! \sim \sqrt{2 \pi r} \left(\frac{r}{e} \right)^r$, we see that
\begin{equation*}
\begin{aligned}
\norm{c_m(t)} & \le \ee^{t \mu(A_0) } \sum\limits_{r=k}^m \frac{ r^r (Nta)^r}{\ee \left(\frac{r}{\ee} \right)^r(r-1)!}\norm{u_0}
\leq \ee^{t \mu(A_0) }  \sum\limits_{r=k}^\infty \frac{(\ee N t a)^r }{ \ee (r-1)! }\norm{u_0} \\
& = \ee^{t \mu(A_0)-1 }(\ee N t a)^k \sum_{r=0}^\infty \frac{r!}{(r+k-1)!}\frac{(\ee Nta)^r}{r!}\norm{u_0} \\
& \le \frac{\ee^{t (\mu(A_0) + \ee N a)-1}(\ee N t a)^k }{(k-1)!}\norm{u_0},
\end{aligned}
\end{equation*}
where in the last inequality we use
%$\sum_{r=1}^\infty \frac{(\ee Nta)^{r-1}}{(r-1)!} =\ee^{\ee Nta}$ and
$\frac{r!}{(r+k-1)!} \le \frac{1}{(k-1)!}$.
\end{proof}

 \end{document}